\documentclass{amsart}
\usepackage{amsmath,amsthm,amsfonts,amssymb,enumerate}
\usepackage{mathrsfs}
\usepackage{tikz}
\usetikzlibrary{through,calc}
\usepackage{color}
\usepackage[
	hyperindex,
	pagebackref,
	breaklinks=true,
	extension=pdf,
	bookmarks=false,
	plainpages=false,
	colorlinks,
	linkcolor=linkblue,
	citecolor=linkblue,
	urlcolor=linkblue,
	pdfmenubar=true,
	pdftoolbar=true,
	pdfpagelabels,
	pdfpagelayout=SinglePage,
	pdfview=Fit,
	pdfstartview=Fit
]{hyperref}
\definecolor{linkblue}{rgb}{0,0.2,0.6}
\usepackage[all]{xy}
\usepackage[alphabetic,lite]{amsrefs}
\newtheorem{theorem}{Theorem}[section]
\newtheorem{proposition}[theorem]{Proposition}
\newtheorem{corollary}[theorem]{Corollary}
\newtheorem{conjecture}[theorem]{Conjecture}
\newtheorem{lemma}[theorem]{Lemma}
\theoremstyle{definition}

\newtheorem{example}[theorem]{Example}
\newtheorem{definition}[theorem]{Definition}
\newtheorem{remark}[theorem]{Remark}
\newtheorem{strategy}[theorem]{Strategy}
\newtheorem{algorithm}[theorem]{Algorithm}

\newcommand{\ov}[1]{\overline{#1}}
\newcommand{\op}[1]{\operatorname{#1}}
\newcommand{\ovop}[1]{\ov{\op{#1}}}

\newcommand{\PP}{\mathbb{P}}
\newcommand{\QQ}{\mathbb{Q}}
\newcommand{\CC}{\mathbb{C}}

\newcommand{\ZZ}{\mathbb{Z}}

\newcommand{\cO}{\mathcal{O} }

\def\Mzn{\ovop{M}_{0,n} }
\def\Mza{\ovop{M}_{0,A} }
\def\Ln{\ovop{L}_{n} }
\def\Lnt{\ovop{L}_{n-2} }
\def\CLn{\mathrm{Chow}_{1}(\Ln) }

\begin{document}

\pagenumbering{arabic}
\title{Effective curves on $\Mzn$ from group actions}
\date{\today}

\author{Han-Bom Moon}
\address{Department of Mathematics, Fordham University, Bronx, NY 10458, USA}
\email{hmoon8@fordham.edu}

\author{David Swinarski}
\address{Department of Mathematics, Fordham University, New York, NY 10023, USA}
\email{dswinarski@fordham.edu}

\begin{abstract}We study new effective curve classes on the moduli
  space of stable pointed rational curves given by the fixed loci of
  subgroups of 
  the permutation group action. We compute their numerical classes and
  provide a strategy for writing them as effective linear combinations of F-curves, using Losev-Manin spaces and toric degeneration of curve classes.
\end{abstract}

\maketitle

\section{Introduction}

One of the central problems in the birational geometry of a projective
variety $X$ is determining its cone of effective curves
$\ovop{NE}_{1}(X)$. In the minimal model
program, this is the first step toward understanding and classifying
all the contractions of $X$, that is, projective morphisms from $X$ to other varieties.

Let $\Mzn$ be the moduli space of stable $n$-pointed rational
curves. Because Kapranov's construction (\cite[Theorem 4.3.3]{Kap93b})
is very similar to a blow-up construction of a toric variety, many
people wondered if the birational geometry of $\Mzn$ might be similar
to that of toric varieties. For instance, the cone of effective cycles
of a toric variety is generated by its torus invariant boundaries;
$\Mzn$ was conjectured to have a similar property. 

\begin{conjecture}[\protect{\cite[Question 1.1]{KM96}}]\label{conj:Fconjecture}
The cone of $k$-dimensional effective cycles is generated by $k$-dimensional intersections of boundary divisors, for $1 \le k \le \dim \Mzn - 1$. 
\end{conjecture}

But in the last decade, there have been several striking results
showing 
that the divisor theory of $\Mzn$ is more complicated than that of toric varieties. For instance, there are non-boundary type extremal effective divisors on $\Mzn$ (\cite{Ver02, CT13a, DGJ14}). Furthermore, very recently Castravet and Tevelev showed that $\Mzn$ is not a Mori dream space for large $n$ (\cite{CT13b}). 

On the other hand, for effective curve classes, there are few known
results in the literature. In \cite{KM96}, Keel and McKernan proved
Conjecture \ref{conj:Fconjecture} for cycles of dimension 
$k = 1$ and $n \le 7$, but it is still unknown for $n > 7$. Conjecture
\ref{conj:Fconjecture} for $k=1$ is now widely known as the
\emph{F-Conjecture}, and one-dimensional intersections of boundary
divisors are called \emph{F-curves}. Keel and McKernan also showed
that if there is any other extremal ray $R$ of the curve cone
$\ovop{NE}_{1}(\Mzn)$ and $\ovop{NE}_{1}(\Mzn)$  is not round at $R$,
then $R$ is generated by a rigid curve intersecting the interior, $\op{M}_{0,n}$ (See \cite[Theorem 2.2]{CT12} for a proof).
This has motivated several researchers 
to search for rigid curves on $\Mzn$ to find a potential counterexample to the
F-Conjecture. Castravet and Tevelev constructed rigid curves by
applying their hypergraph construction in \cite{CT12}. There is a
slightly weaker notion of the rigidity of curves (so-called rigid
maps). In \cite{CT12} and \cite{Che11}, Castravet, Tevelev, and Chen  constructed two
types of examples of rigid maps. But all of these examples are
numerically equivalent to effective linear combinations of F-curves,
thus they do not give counterexamples to the F-conjecture. To our
knowledge, these examples, the F-curves, and some curve
classes that arise from obvious families of point configurations are the only explicit examples of effective curves on $\Mzn$
in the literature.

%\medskip

\subsection{Aim of this paper}

This project began because we wanted to study the geometric
and numerical properties of some new effective curves on $\Mzn$
that arise from a finite group action.

There is a natural $S_{n}$-action on $\Mzn$
permuting the marked points. Let $G$ be a subgroup of $S_{n}$. Let $\Mzn^{G}$ be the union of irreducible components of the $G$-fixed locus that intersect the interior $\mathrm{M}_{0,n}$. If we impose certain numerical conditions on $G$, then $\Mzn^{G}$ becomes an irreducible curve on $\Mzn$. 

The computation of the numerical class of $\Mzn^{G}$ (or equivalently,
its intersection with boundary divisors) is elementary. But
surprisingly, there has been no study of these curve classes on
$\Mzn$. We believe that the reason is that even though it is
straightforward to compute the numerical class of such a curve, it is
difficult to
determine whether the curve is numerically equivalent to an
\emph{effective} linear combination of F-curves. Indeed, it is a
difficult computation to find an actual effective linear combination
of F-curves, and that led to the main result of our paper. 

%\medskip

\subsection{Main result}

The main result of this paper is not a single theorem, but a method to
approach this computational problem using Losev-Manin spaces $\Ln$
(\cite{LM00}) and toric degenerations. Losev-Manin spaces $\Ln$ are 
special cases of Hassett's moduli spaces of stable weighted pointed
rational curves (\cite{Has03}). As moduli spaces, they parametrize
pointed chains of rational curves, and they are contractions of
$\ovop{M}_{0,n+2}$. A significant geometric property of $\Ln$ is that
it is the closest toric variety to $\ovop{M}_{0,n+2}$ among Hassett's
spaces. $\Ln$ is a toric variety whose corresponding polytope is the
permutohedron of dimension $n-1$ (\cite[Section 7.3]{GKZ08}). We give a method to
compute a toric degeneration of an effective curve class on $\Ln$. The computation is a result of an interesting interaction between the moduli theoretic interpretation of $\Ln$ and the combinatorial structure of the permutohedron. 

For any effective curve class $C$ on $\Mzn$, we are able to compute
the numerical class of its image $\rho(C)$ for $\rho : \Mzn\to
\Lnt$. By using the toric degeneration method, we can find an
effective linear combination of one dimensional toric boundaries
representing $\rho(C)$. Since each toric boundary component is the
image of a unique F-curve, it is an ``approximation'' of the effective
linear combination for $C$. By taking the proper transform, we have a (not necessarily effective) 
linear combination of F-curves for $C$ on $\Mzn$. To find an effective
linear combination, we use a computational strategy described in Section \ref{sec:efflincomb}. 

It is significant to note that our computational strategy doesn't use any special properties of the
finite group action used to define the curves $\Mzn^{G}$, and thus we believe that the approach using
Losev-Manin spaces and toric degenerations may also be applicable to
other curve classes not of the form $\Mzn^{G}$.  For example, 
this could give a second approach to analyzing hypergraph
curves whose classes are computed in \cite{CT12} using the technique of arithmetic breaks. 

In this paper, we study $\Mzn^{G}$ when $G$ is either a cyclic group
or a dihedral group. 

%\medskip

\subsection{Cyclic group case}

If $G$ is a cyclic group, the geometric and numerical properties of
$\Mzn^{G}$ are relatively easy to prove without using technical tools
such as toric degeneration.  We prove the following theorem for arbitrary $n$.

\begin{theorem}\label{thm:mainthmcyclic}
Let $G = \langle\sigma\rangle$ be a cyclic group.
\begin{enumerate}
\item (Lemma \ref{lem:nonempty}) The invariant subvariety $\Mzn^{G}$ is nonempty if and only if
  $\sigma$ is \emph{balanced} (see Definition \ref{def:balanced}).
\item (Lemmas \ref{lem:irreducible}, \ref{lem:isomorphictoP1}) The invariant subvariety $\Mzn^{G}$ is irreducible and when it is a curve, it is isomorphic to $\PP^{1}$. 
\item (Theorem \ref{thm:cycliceffectivesumofFcurves}) In this case, $\Mzn^{G}$ is numerically equivalent to a linear combination of F-curves such that all coefficients are one. 
\item (Theorem \ref{thm:cycliccasemovable}) $\Mzn^{G}$ is movable. 
\end{enumerate}
\end{theorem}

\subsection{Dihedral group case}
If the group $G$ is not cyclic, then in general the
intersection number with the canonical divisor is positive, so it 
is possible that $\Mzn^{G}$ is rigid (Section
\ref{sec:dihedral}). Thus by using this idea, we might find a new
extremal ray of $\ovop{NE}_{1}(\Mzn)$. But we checked all such curves
for $n \leq 12$, and none of them is a counterexample to the
F-conjecture.  In the paper we include two concrete examples
on $\ovop{M}_{0,9}$ and $\ovop{M}_{0,12}$.  These are introduced in
Examples \ref{ex:dihedralcurve} and \ref{ex:dihedraln=12} and
completed in sections 
\ref{sec:efflincomb} and \ref{sec:example}.

\subsection{The M0nbar package for Macaulay2}
We have written a great deal of code in \texttt{Macaulay2} (\cite{M2}) over a period of
many years.  For this project, we collected our work in a package
called \texttt{M0nbar} for \texttt{Macaulay2}.  The package code is
available at the second author's website:
\begin{center}
http://faculty.fordham.edu/dswinarski/M0nbar/
\end{center}
This package was used to check many of the calculations in Sections
\ref{sec:efflincomb} and \ref{sec:example}.  We have posted code
samples for these calculations and some others on the website:  
\begin{center}
http://faculty.fordham.edu/dswinarski/invariant-curves/
\end{center}

%%\medskip
\subsection{Structure of the paper}

Here is an outline of this paper. In Section \ref{sec:generality}, we
introduce several definitions we will use in this paper. In Section
\ref{sec:cyclic}, we prove several geometric and numerical properties
for invariant curves for cyclic groups. In Section
\ref{sec:efflincombcyclic}, we prove Theorem
\ref{thm:mainthmcyclic}. The dihedral group case is explained in
Section \ref{sec:dihedral}. The main method of this paper, using
Losev-Manin spaces, is described in Section \ref{sec:LosevManin}. In
Section \ref{sec:efflincomb}, we give a computational strategy to find
an effective $\ZZ$-linear combination from a non-effective
combination. In Section \ref{sec:example}, we compute
an example on $\ovop{M}_{0,12}$.

\subsection{Acknowledgements}  We would like to thank Angela Gibney
for teaching us about the nonadjacent basis.

\medskip

\section{Loci in $\Mzn$ fixed by a finite group}\label{sec:generality}

We work over the complex numbers $\CC$ throughout the paper. 

Consider the natural $S_{n}$-action on $\Mzn$ permuting the $n$ marked points. 

\begin{definition}
Fix a subgroup $G \le S_{n}$ and let $\Mzn^{G}$ be the union of those 
irreducible components of the $G$-fixed locus for the induced
$G$-action on $\Mzn$ that intersect the interior $\mathrm{M}_{0,n}$. Then an irreducible component of $\Mzn^{G}$ is a subvariety of $\Mzn$.
\end{definition}

\begin{remark}
One motivation for studying these loci is the following.   There are
several different possible descriptions of Keel-Vermeire divisors
(\cite{Ver02}). One of them is using $\Mzn^{G}$: a Keel-Vermeire divisor is the case that $G$ is a cyclic group of order two generated by $(12)(34)(56)$. 
\end{remark}

\begin{remark}
In general, there are several irreducible components of the $G$-fixed
loci which are contained in the boundary. For example, see Example
\ref{ex:ontheboundary}.  That is why in the definition we choose only
those components that meet the interior.
\end{remark}

Let $n \ge 3$. For a subgroup $G \subset S_{n}$, consider $(\PP^{1}, x_{1}, \cdots, x_{n}) \in \Mzn^{G}$. Then for each $\sigma \in G$, there is $\phi_{\sigma} \in \mathrm{Aut}(\PP^{1}) = \mathrm{PGL}_{2}$ such that $\phi_{\sigma}(x_{i}) = x_{\sigma(i)}$. It defines a group representation $\phi: G \to \mathrm{PGL}_{2}$. This representation is faithful, because if $\phi_{\sigma} = \mathrm{id} \in \mathrm{PGL}_{2}$, $x_{i} = \phi_{\sigma}(x_{i}) = x_{\sigma(i)}$ thus $\sigma = \mathrm{id} \in S_{n}$. We can conclude that $\Mzn^{G}$ is nonempty only if there exists a faithful representation $\phi : G \to \mathrm{PGL}_{2}$. 

A finite subgroup of $\mathrm{PGL}_{2}$ is one of following.
\begin{itemize}
\item A finite cyclic group $C_{k}$. 
\item A dihedral group $D_{k}$. 
\item $A_{4}, S_{4}, A_{5}$. 
\end{itemize}

Moreover, any $\psi \in \mathrm{PGL}_{2}$ with finite order $r$ is, up
to conjugation, a rotation along a pivotal axis on $\PP^{1} \cong
S^{2}$ by the angle $\frac{2\pi}{r}$. Therefore it has two fixed points,
and except them, all other orbits have length $r$. Thus, to obtain a faithful representation, for all $\sigma \in G - \{e\}$, the number of elements in $\mathrm{Stab}_{\sigma} := \{ i \in [n] | \sigma(i) = i\}$ must be at most two. 

These restrictions already give all possible $\sigma \in S_{n}$ with nonempty fixed locus $\Mzn^{\langle \sigma \rangle}$.

\begin{lemma}\label{lem:nonempty}
Let $\sigma \in S_{n}$ and let $\sigma = \sigma_{1}\sigma_{2}\cdots \sigma_{k}$ where the right hand side is a product of disjoint nontrivial cycles. Suppose that $\Mzn^{\langle \sigma \rangle}$ is nonempty. If we denote the length of $\sigma_{i}$ by $\ell_{i}$, then $\ell_{1} = \ell_{2} = \cdots = \ell_{k}$ and $n - \sum_{i=1}^{k}\ell_{i} \le 2$. 
\end{lemma}

Conversely, for a $\sigma \in S_{n}$ satisfying the conditions in Lemma \ref{lem:nonempty}, it is easy to find $(\PP^{1}, x_{1}, \cdots, x_{n}) \in \Mzn^{\langle \sigma\rangle}$. 

\begin{definition}\label{def:balanced}
A permutation $\sigma \in S_{n}$ is called \emph{balanced} if we can
write $\sigma$ as a product of disjoint nontrivial cycles
$\sigma_{1}\sigma_{2} \cdots \sigma_{k}$ such that 
\begin{enumerate}
\item the length of all $\sigma_{i}$'s are equal to a fixed $\ell$;
\item $n - k\ell \le 2$.
\end{enumerate}
\end{definition}

\section{Cyclic group cases}\label{sec:cyclic}

In this section, we will consider cyclic group cases. 

\begin{lemma}\label{lem:irreducible}
Let $G = \langle \sigma \rangle$ be a cyclic group of order $r$ such that $\sigma$ is balanced. Let $j$ be the number of trivial (length 1) cycles in $\sigma$. Then:
\begin{enumerate}
\item The dimension of $\Mzn^{G}$ is $\frac{n-j}{r} -1$. 
\item $\Mzn^{G}$ is irreducible.
\end{enumerate}
\end{lemma}

\begin{proof}
Note that $n-j$ marked points of order $r$ can be decomposed into $\frac{n-j}{r}$ orbits and each orbit is determined by a choice of a point of $\PP^{1}$. Thus a point of $\Mzn^{G}$ is determined by isomorphism classes of $\frac{n-j}{r} + 2$ distinct points on $\PP^{1}$, where the last two points are $\sigma$-fixed points. Hence the dimension of $\Mzn^{G}$ is $\frac{n-j}{r}+2-3 = \frac{n-j}{r}-1$. This proves (1). 

Moreover, by the above description, there is a dominant rational map $(\PP^{1})^{\frac{n-j}{r}+2} - \Delta \dashrightarrow \Mzn^{G}$. Therefore $\Mzn^{G}$ is irreducible.
\end{proof}

\begin{example}
\begin{enumerate}
\item For $n = 6$, there are three types of positive-dimensional
  subvarieties. One is codimension one, which is the case $j = 0$ and
  $r = 2$. So the group $G$ is generated by $
  (12)(34)(56)$ or one of its $S_{6}$ conjugates. Thus in this
  case $\Mzn^{G}$ is a Keel-Vermeire divisor. If $j = 2$ and $r = 2$,
  $G$ is generated by $ (12)(34)$ or one of its
  $S_{6}$-conjugates. Finally, if $j = 0$ and $r = 3$, $G $ is
  generated by $ (123)(456)$ or one of its $S_{6}$-conjugates.
\item When $n = 7$, there are two positive dimensional subvarieties. If $j = 1$ and $r = 2$, $\Mzn^{G}$ is two-dimensional. If $j = 1$ and $r = 3$, it is a curve. 
\end{enumerate}
\end{example}

\begin{remark}
A simple consequence is that $\Mzn^{G}$ has codimension at least two
if $n \ge 7$. Indeed, $\frac{n-j}{r} - 1 = n - 4$ has integer
solutions with $0 \le j \le 2$ and $r \ge 2$ only if $n \le
6$. Similarly, $\Mzn^{G}$ has codimension two only if $n \le 8$. So in
general, we are only able to obtain subvarieties with large codimension.
\end{remark}

\begin{lemma}\label{lem:isomorphictoP1}
Let $G =\langle \sigma \rangle$ be a cyclic group where $\sigma$ is a
balanced permutation. If the dimension of $\Mzn^{G}$ is one, then $\Mzn^{G}
\cong \PP^{1}$.
\end{lemma}

\begin{proof}
Since the $n=4$ case is obvious, suppose that $n \ge 5$. Pick one element from each cycle, and let $S$ be the set of them. Also, if there are non-marked $\sigma$-fixed points, then enlarge $S$ to include these $\sigma$-fixed points, too. When $n \ge 5$, it is easy to see that $|S| = 4$. Then there is a morphism $\pi : \Mzn^{G} \to \ovop{M}_{0, S}\cong \PP^{1}$. Then $\pi$ is a regular birational morphism. A birational morphism from a complete curve to a nonsingular complete curve is an isomorphism (see for instance \cite[Proposition III.9.1]{Mum99}).
\end{proof}

\begin{remark}
In general, $\Mzn^{G}$ is a rational variety because there is a
birational map $(\PP^{1})^{k} \dashrightarrow \Mzn^{G}$. It would be
interesting if one can describe the geometry of $\Mzn^{G}$ in terms of
concrete blow-ups and blow-downs of $(\PP^{1})^{k}$.
\end{remark}

\begin{definition}\label{def:Csigma}
For $G = \langle \sigma \rangle$ with a balanced $\sigma$, if $\dim \Mzn^{G} = 1$, we will denote $\Mzn^{G}$ by $C^{\sigma}$.
\end{definition}

In this case, there are exactly two nontrivial disjoint cycles of length $r$. Let
$j$, $r$ be two integers satisfying $0 \le j \le 2$, $r > 1$, $\frac{n-j}{r} =
2$. So $j$ refers the number of fixed marked points, and $r$ is the
length of a general orbit on $\PP^{1}$. We will say that $\sigma$ (or $G$) is \emph{of type $(j, r)$}. 

For any $n \ge 5$, there is a cyclic group $G$ such that $\dim \Mzn^{G} = 1$. Indeed, by taking an appropriate $0 \le j \le 2$, $n-j$ can be an even number $2r$ with $r > 1$. Thus $\frac{n-j}{r} - 1 = 1$ has an integer solution and we are able to find $G$. 

In the next lemma, we show that if $j > 0$, then $C^{\sigma}$ comes from a curve on $\ovop{M}_{0,n-1}$.

\begin{lemma}\label{lem:Crj-1andCrj}
Suppose that for a balanced $\sigma$ with $G = \langle \sigma
\rangle$, $\sigma(i) = i$. Let $G' = \langle \sigma'\rangle$ be the
cyclic subgroup of $S_{n-1}$ generated by
$\sigma':=\sigma|_{[n]-\{i\}}$. Consider $C^{\sigma'} \subset
\ovop{M}_{0,[n]-\{i\}} \cong \ovop{M}_{0,n-1}$. For the forgetful map $\pi: \Mzn \to \ovop{M}_{0,n-1}$, $\pi|_{C^{\sigma}} : C^{\sigma} \to C^{\sigma'}$ is an isomorphism.
\end{lemma}

\begin{proof}
It is straightforward to check that $\pi(C^{\sigma}) = C^{\sigma'}$
and that $\pi$ restricts to a birational map on $C^{\sigma}$. Thus $C^{\sigma} \to C^{\sigma'}$ is an isomorphism.
\end{proof}

\begin{example}\label{ex:ontheboundary}
Let $G = \langle (12)(34)\rangle$ and $n = 6$. Consider $X_{1} =
\PP^{1}$ with 4 marked points $x_{1}$, $x_{2}$, $x_{3}$, and $x_{4}$
such that $G$-invariant. Then there are two fixed points $p, q$ on
$X_{1}$. Let $X_{2} = \PP^{1}$ have three marked points $x_{5}$,
$x_{6}$, $r$. Consider the gluing $X = X_{1} \cup X_{2}$ along $p$ and
$r$. Then this is a $G$-invariant curve, and there are 1-dimensional
moduli $C \subset \ovop{M}_{0,6} - \mathrm{M}_{0,6}$ of these
curves. $C$ is disjoint from $C^{(12)(34)}$, because for every
degenerated curve in $C^{(12)(34)}$, two fixed marked points $x_{5}$
and $x_{6}$ are on the spine (for instance when $x_{1}$ approaches
$x_{3}$ or $x_{4}$) or on distinct tails (for example if $x_{1}$
approaches one of $\sigma$-fixed points on $\PP^{1}$). 

This example shows that, in general, the fixed point locus of $G$ has
extra irreducible components contained in the boundary of $\Mzn$.
\end{example}

Since the classes of boundary divisors span $\mathrm{N}^{1}(\Mzn,\mathbb{Q})$ (\cite[p.550]{Kee92}), to
obtain the numerical class of $C^{\sigma}$, it suffices to know the
intersection numbers of $C^{\sigma}$ with boundary divisors (consult \cite[Section 2]{Moo13} for notations). 

\begin{proposition}\label{prop:intersectionboundary}
Let $\sigma \in S_{n}$ be a balanced permutation of type $(j, r)$ with two nontrivial orbits and let $\sigma_{1}$, $\sigma_{2}$ be two non-trivial disjoint cycles in $\sigma$. Let $F$ be the set of $\sigma$-invariant marked points. 
\begin{enumerate}
\item Let $I = \{h,i\}$ where $h \in \sigma_{1}$ and $i \in \sigma_{2}$. Then $C^{\sigma}\cdot D_{I} = 1$.
\item Suppose that $F = \emptyset$. For $I = \sigma_{\ell}$, $C^{\sigma} \cdot D_{I} = 2$.
\item Suppose that $F = \{a\}$. For $I = \sigma_{\ell}$ or $I = \sigma_{\ell} \cup \{a\}$, $C^{\sigma} \cdot D_{I} = 1$.
\item Suppose that $F = \{a, b\}$. For $I = \sigma_{\ell} \cup \{a\}$, $C^{\sigma}\cdot D_{I} = 1$. 
\item Except for the above cases, all other intersection numbers are zero.
\item If $n \ge 6$, $C^{\sigma} \cdot D_{2} = r^{2}$ and $C^{\sigma} \cdot D_{\lfloor \frac{n}{2}\rfloor} = 2$. All other intersections are zero.
\item If $n = 5$ (so $j = 1$ and $r = 2$), then $C^{\sigma} \cdot D_{2} = 6$. 
\end{enumerate}
\end{proposition}

\begin{proof}
We count all points on $C^{\sigma}$ that parametrize 
singular curves. They arise when some of the marked points
(equivalently, some of the orbits) collide. There are two marked orbits (orbits consisting of marked points) of order $r$ and $j$ marked orbits of order one. 

First of all, consider the case that two marked orbits collide. In particular, suppose that $x_{h}$ for $h \in \sigma_{1}$ approaches $x_{i}$ for $i \in \sigma_{2}$, where $x_{h}$ (resp. $x_{i}$) is the $h$-th (resp. $i$-th) marked point. Then simultaneously $x_{\sigma^{k}(h)} = x_{\sigma_{1}^{k}(h)}$ approaches $x_{\sigma^{k}(i)} = x_{\sigma_{2}^{k}(i)}$, with a constant rate of speed. Thus the stable limit is on $\cap_{0 \le k < r} D_{\{\sigma^{k}(h), \sigma^{k}(i)\}}$, which parametrizes a curve with $r$ tails such that on each tail there are two marked points $x_{\sigma^{k}(h)}$ and $x_{\sigma^{k}(i)}$. This proves (1). 

For $i \in \sigma_{1}$, if $x_{i}$ approaches one of two $\sigma$-fixed points on $\PP^{1}$, then we have another singular stable limit. If $x_{i}$ goes to a non-marked point, then the limit is on $D_{I}$ where $I = \sigma_{1}$. If $x_{i}$ goes to a fixed marked point $x_{a}$, then the limit is on $D_{I\cup \{a\}}$. So we obtain parts (2), (3), and (4). 

Except those cases above, all other points on $C^{\sigma}$ parametrize smooth curves. So we have part (5).

If $n \ge 7$, then $r \ge 3$. In this case, item (1) is the only case
where $C^{\sigma} \cdot D_{I} > 0$ with $|I| = 2$. Because for a fixed $i$, there are $r$ different possible choices of $j$, we have $C^{\sigma} \cdot D_{2} = r^{2}$. There are two additional degenerations corresponding to the case that $x_{i}$ approaches one of two fixed points. In any case of (2), (3), and (4), it is straightforward to check that $|I| = \lfloor \frac{n}{2}\rfloor$. Therefore $C^{\sigma}\cdot D_{\lfloor \frac{n}{2}\rfloor} = 2$. The case of $n = 5, 6$ are obtained by a simple case by case analysis with the same idea. So we have (6) and (7). 
\end{proof}

\begin{corollary}\label{cor:intersectioncanonicaldivisor}
Let $\sigma \in S_{n}$ be a balanced permutation of type $(j, r)$. Then 
\[
	C^{\sigma}\cdot K_{\Mzn} = -4+j.
\]
\end{corollary}

\begin{proof}
The numerical class of the canonical divisor $K_{\Mzn}$ is given by 
\begin{equation}\label{eqn:canonicaldivisor}
K_{\Mzn} = \sum_{k=2}^{\lfloor \frac{n}{2}\rfloor}\left(-2 + \frac{k(n-k)}{n-1}\right)D_{k}
\end{equation}
(\cite[Proposition 1]{Pan97}). By using Proposition
\ref{prop:intersectionboundary} and formula \eqref{eqn:canonicaldivisor}, we can compute the intersection number.
\end{proof}

Since $\psi \equiv K_{\Mzn}+2D$ (\cite[Lemma 2.9]{Moo13}), it is immediate to get the following result.

\begin{corollary}
Let $\sigma \in S_{n}$ be a balanced permutation of type $(j, r)$. Then 
\[
	C^{\sigma} \cdot \psi = 2r^{2}+j.
\]
\end{corollary}

\section{The curve class $C^{\sigma}$ as an effective linear combination of F-curves}\label{sec:efflincombcyclic}
In this section, we show two facts.  First, the curve class $C^{\sigma}$ in Definition \ref{def:Csigma} is an effective $\ZZ$-linear combination of F-curves. Second, it is movable, so it is in the dual cone of the effective cone of $\Mzn$. 

\subsection{The nonadjacent basis and its dual basis}

For $\mathrm{Pic}(\Mzn)_{\QQ} \cong \mathrm{H}^2(\Mzn, \QQ)$, there is a basis
due to Keel and Gibney called the \emph{nonadjacent} basis. Let $G_{n}$ be a cyclic graph
with $n$ vertices $[n] := \{1,2,\cdots, n\}$ labeled in that order. For a subset $I \subset [n]$, let $t(I)$ be the number of connected components of the subgraph generated by vertices in $I$. A subset $I$ is called \emph{adjacent} if $t(I) = 1$. Since $G_{n}$ is cyclic, if $t(I) = k$, then $t(I^{c}) = k$.

\begin{proposition}[\cite{Car09}*{Proposition 1.7}] \label{prop:nonadjacentbasis}
Let $\mathcal{B}$ be the set of boundary divisors $D_{I}$ for $I \subset [n]$ with $t(I) \ge 2$. Then $\mathcal{B}$ forms a basis of $\mathrm{Pic}(\Mzn)_{\QQ}$.
\end{proposition}

\begin{definition}\label{def:nonadjacentbasis}
The set $\mathcal{B}$ in Proposition \ref{prop:nonadjacentbasis} is called the \emph{nonadjacent basis}. 
\end{definition}

Since there is an intersection pairing $\mathrm{H}^{2}(\Mzn, \QQ) \times
\mathrm{H}_{2}(\Mzn, \QQ) \to \QQ$, we obtain a basis of  $\mathrm{H}_2(\Mzn, \QQ)$ dual
to the nonadjacent basis. Indeed, many vectors in this basis are F-curves. 

\begin{example}
On $\ovop{M}_{0,5}$, the nonadjacent basis is
\[
	\{D_{\{1,3\}}, D_{\{1,4\}}, D_{\{2,4\}}, D_{\{2,5\}}, D_{\{3,5\}}\}.
\]
It is straightforward to check that the dual basis is (in the corresponding order)
\[
	\{F_{\{1,2,3,45\}}, F_{\{1,4,5,23\}}, F_{\{2,3,4,15\}}, 
	F_{\{1,2,5,34\}}, F_{\{3,4,5,12\}}\}.
\]
So all of the dual elements are F-curves.
\end{example}

\begin{example}\label{ex:dualofnonadjacentonn=6}
On $\ovop{M}_{0,6}$, the nonadjacent basis is
\[
	\{D_{\{1,3\}}, D_{\{1,4\}}, D_{\{1,5\}}, D_{\{2,4\}}, D_{\{2,5\}}, 	D_{\{2,6\}}, D_{\{3,5\}}, D_{\{3,6\}}, D_{\{4,6\}},
\]
\[
	D_{\{1,2,4\}}, D_{\{1,2,5\}}, D_{\{1,3,4\}}, D_{\{1,3,5\}}, 
	D_{\{1,3,6\}}, D_{\{1,4,5\}}, D_{\{1,4,6\}}\}.
\]
The dual basis is
\[
	\{F_{\{1,2,3,456\}}, F_{\{1,4,23,56\}}, F_{\{1,5,6,234\}}, 
	F_{\{2,3,4,156\}}, F_{\{2,5,16,34\}}, F_{\{1,2,6,345\}}, 
\]
\[
	F_{\{3,4,5,126\}}, F_{\{3,6,12,45\}}, F_{\{4,5,6,123\}},
	F_{\{3,4,12,56\}}, F_{\{5,6,12,34\}}, F_{\{1,2,34,56\}}, 
\]
\[
	F_{\{5,6,13,24\}}+F_{\{1,2,3,456\}}+F_{\{2,3,4,156\}}-
	F_{\{2,3,16,45\}},
\]
\[
	F_{\{2,3,16,45\}}, F_{\{1,6,23,45\}}, F_{\{4,5,16,23\}}\}.
\]
Except for the curve dual to $D_{\{1,3,5\}}$, all other vectors in the dual basis are F-curves.
\end{example}

%{\color{red} Change I,J below to B,G? HB: I think it would be better to keep the current notation using I, J.}

\begin{proposition}\label{prop:dualFcurve}
Let $D_{I}$ be a boundary divisor in the nonadjacent basis. The dual of $D_{I}$ is an F-curve if and only if $t(I) = 2$. In this case, if $I_{1} \sqcup I_{2} = I$ and $J_{1} \sqcup J_{2} = I^{c}$ are the decompositions of $I$ and $I^{c}$ into two connected sets, then the dual of $D_{I}$ is $F_{I_{1}, I_{2}, J_{1}, J_{2}}$.
\end{proposition}

\begin{proof}
Since $I$ is not connected, $t(I) \ge 2$. For $t(I) > 2$, let $I_{1}, I_{2}, \cdots, I_{t(I)}$ be the connected components of $I$. To obtain $D_{I} \cdot F_{A_{1}, A_{2},A_{3}, A_{4}} = 1$, we need $A_{i} \sqcup A_{j} = I$ for two distinct $i, j \in \{1,2,3,4\}$. Then at least one of $A_{i}$ must be disconnected. So $D_{A_{i}}$ is in the nonadjacent basis and $D_{A_{i}} \cdot F_{A_{1}, A_{2}, A_{3}, A_{4}}= -1$. Therefore the dual element is not an F-curve.

Now suppose that $t(I) = 2$. Let $I = I_{1}\sqcup I_{2}$ and $I^{c} = J_{1} \sqcup J_{2}$ be the decompositions of $I$ and $I^{c}$ into connected components. Then for $D_{K} \in \mathcal{B}$, we have a nonzero intersection $D_{K} \cdot F_{I_{1}, I_{2}, J_{1}, J_{2}}$ only if $K$ is one of $I_{1}, I_{2}, J_{1}, J_{2}, I_{1}\sqcup I_{2}, I_{1}\sqcup J_{1}, I_{1} \sqcup J_{2}$ or their complements. But except $I = I_{1}\sqcup I_{2}$ (equivalently, its complement $I^{c} = J_{1} \sqcup J_{2}$), all of them are connected so only $D_{I}$ is in $\mathcal{B}$. Thus $F_{I_{1}, I_{2}, J_{1}, J_{2}}$ is the dual element of $D_{I}$. 
\end{proof}

In general, the dual curve for $D_{I}$ with $t(I) > 2$ is a complicated non-effective $\ZZ$-linear combination of F-curves. We give an inductive algorithm to find the combination.

\begin{proposition}
For $D_{I}$ with $t(I) = t \ge 3$, let $I = I_{1}\sqcup I_{2}\sqcup \cdots \sqcup I_{t}$ and $I^{c} = J_{1}\sqcup J_{2}\sqcup \cdots \sqcup J_{t}$ be the decompositions into connected components, in the circular order of $I_{1}, J_{1}, I_{2}, J_{2}, \cdots, I_{t}, J_{t}$. Then the dual basis of $D_{I}$ is a $\ZZ$-linear combination of F-curves. The linear combination can be computed inductively. 
\end{proposition}

\begin{proof}
If $t(I) = 3$, it is straightforward to check that the dual curve for $D_{I}$ is 
\[
	F_{I_{1}\sqcup I_{2}, J_{1}\sqcup J_{2}, I_{3}, J_{3}}
	+ F_{I_{1},J_{1},I_{2},J_{2}\sqcup I_{3}\sqcup J_{3}}
	+ F_{J_{1},I_{2},J_{2},I_{1}\sqcup I_{3}\sqcup J_{3}}
	- F_{J_{1},I_{2},I_{1}\sqcup J_{3},J_{2}\sqcup I_{3}}.
\]

Consider the F-curve $F_{I_{1}\sqcup I_{2}, J_{1}\sqcup J_{2}, I_{3}\sqcup
  I_{4}\sqcup \dots \sqcup I_{t}, J_{3}\sqcup J_{4}\sqcup \cdots \sqcup
  J_{t}}$. It intersects positively with $D_{I_{1}\sqcup J_{1}\sqcup
  I_{2}\sqcup J_{2}}$, $D_{I}$, $D_{I_{1}\sqcup I_{2}\sqcup J_{3}\sqcup
  J_{4}\sqcup \cdots \sqcup J_{t}}$, and negatively with $D_{I_{1}\sqcup
  I_{2}}$, $D_{J_{1}\sqcup J_{2}}$, $D_{I_{3}\sqcup I_{4}\sqcup \cdots \sqcup
  I_{t}}$, and $D_{J_{3}\sqcup J_{4}\sqcup \cdots \sqcup J_{t}}$. Except
them, all other intersection numbers are zero. Note that $I_{1} \sqcup
J_{1} \sqcup I_{2} \sqcup J_{2}$ is connected so $D_{I_{1}\sqcup J_{1} \sqcup
  I_{2} \sqcup J_{2}} \notin \mathcal{B}$. For all other boundary
divisors above with nonzero intersection numbers, the numbers of
connected components of corresponding subsets of $[n]$ are strictly
less than $t$, because $I_{1}\sqcup J_{t}$ is a connected set. Let
$E_{I}$, $E_{I_{1}\sqcup I_{2}\sqcup J_{3}\sqcup J_{4}\sqcup \cdots \sqcup
  J_{t}}$, $E_{I_{1}\sqcup I_{2}}$, $E_{J_{1}\sqcup J_{2}}$, $E_{I_{3}\sqcup
  I_{4}\sqcup \cdots \sqcup I_{t}}$, and $E_{J_{3}\sqcup J_{4}\sqcup \cdots
  \sqcup J_{t}}$ be the dual elements which are explicit $\ZZ$-linear
combinations of F-curves by the induction hypothesis. Then 
\[
	F_{I_{1}\sqcup I_{2}, J_{1}\sqcup J_{2}, 
	I_{3}\sqcup I_{4}\sqcup \dots \sqcup I_{t}, 
	J_{3}\sqcup J_{4}\sqcup \cdots \sqcup J_{t}}
	- E_{I_{1}\sqcup I_{2}\sqcup J_{3}\sqcup J_{4}\sqcup \cdots \sqcup J_{t}}
\]
\[
	+ E_{I_{1}\sqcup I_{2}} + E_{J_{1}\sqcup J_{2}} + 
	E_{I_{3}\sqcup I_{4}\sqcup \cdots \sqcup I_{t}} + 
	E_{J_{3}\sqcup J_{4}\sqcup \cdots \sqcup J_{t}}
\]
is the dual basis of $D_{I}$.
\end{proof}

\begin{remark}
The rank of $\mathrm{Pic}(\Mzn)_{\QQ}$ is $2^{n-1}-{n \choose 2} -
1$. On the other hand, the number of boundary divisors $D_{I}$ with
$t(I) = 2$ is ${n \choose 4}$. Therefore if $n$ is large, most dual
vectors are not F-curves.
\end{remark}

\subsection{Writing  $C^{\sigma}$ as an effective linear combination of F-curves}

\begin{theorem}\label{thm:cycliceffectivesumofFcurves}
Let $\sigma \in S_{n}$ be a balanced permutation which is the product
of two nontrivial disjoint cycles $\sigma_{1}$ and $\sigma_{2}$. Then
$C^{\sigma}$ can be written as an effective sum of F-curves whose
nonzero coefficients are all one. 
\end{theorem}

\begin{proof}
By applying the left $S_{n}$-action to $S_{n}$, we may
assume that $\sigma = \sigma_{1}\sigma_{2} = (1,2,\cdots, r)(r+1, r+2,
\cdots, 2r)$. Let $S$ be the set of boundary divisors that intersect $C^{\sigma}$ nontrivially. We claim that for any $D_{I} \in S$, $t(I) \le 2$. Indeed, for $D_{I} \in S$ with $|I| = 2$, then $t(I) \le 2$ is obvious. If $|I| > 2$, then $I$ is one of $\sigma_{i}$, which is a connected set, or possibly the union of $\sigma_{i}$ and an element of $[n] - (\sigma_{1} \sqcup \sigma_{2})$. Since $[n]-(\sigma_{1} \sqcup \sigma_{2})$ has at most two elements, the union also has at most two connected components. 

For each nonadjacent $D_{I}\in S$, the dual element is an F-curve by
Lemma \ref{prop:dualFcurve}.   Thus $C^{\sigma}$ is the effective linear combination of dual elements for nonadjacent $D_{I}\in S$. Finally, the intersection number $C^{\sigma}\cdot D_{I}$ is greater than one only if $j = 0$ and $I = \sigma_{i}$. But in this case $I$ is connected so it does not affect the linear combination. 
\end{proof}

\begin{corollary}
The number of F-curves appearing in the effective linear combination
of $C^{\sigma}$ described by Theorem \ref{thm:cycliceffectivesumofFcurves} is $r^{2}-2+j$. 
\end{corollary}

\begin{proof}
We may assume that $G = \langle (1,2,\cdots, r)(r+1, r+2, \cdots, 2r)\rangle$. Note that $C^{\sigma}$ intersects $r^{2}+2$ boundary divisors by Proposition \ref{prop:intersectionboundary}. If $j = 0$ (so $n = 2r$), then $C^{\sigma} \cdot D_{\{1,2,\cdots, r\}} = 2$, $C^{\sigma}\cdot D_{\{1,2r\}} = C^{\sigma}\cdot D_{\{r,r+1\}} = 1$ and other intersecting boundaries are all nonadjacent. Thus on the effective linear combination above, $r^{2}-2$ F-curves appear. If $j = 1$ and $n = 2r+1$, then $C^{\sigma}\cdot D_{\{1,2,\cdots, r\}} = C^{\sigma}\cdot D_{\{r+1,r+2,\cdots, 2r\}} = C^{\sigma}\cdot D_{\{r,r+1\}} = 1$ and other intersecting boundaries are nonadjacent so there are $r^{2}-1$ F-curves on the effective linear combination. The case of $j=2$ is similar.
\end{proof}

\begin{example}
Let $\sigma = (12)(34) \in S_{6}$. By Proposition \ref{prop:intersectionboundary}, $C^{\sigma}$ intersects nontrivially 
\[
	D_{\{1,3\}}, D_{\{1,4\}}, D_{\{2,3\}}, D_{\{2,4\}}, D_{\{1,2,5\}}, 
	D_{\{1,2,6\}}
\]
and the intersection numbers are all one. Among these divisors, $D_{\{2,3\}}$ and $D_{\{1,2,6\}}$ are adjacent divisors. Thus 
\[
	C^{\sigma} \equiv 
	F_{\{1,2,3,456\}} + F_{\{1,4,23,56\}} + F_{\{2,3,4,156\}}
	+ F_{\{5,6,12,34\}}.
\]
\end{example}

\begin{example}\label{ex:C30G}
Let $\sigma = (123)(456) \in S_{6}$. Then by using Proposition
\ref{prop:intersectionboundary} and the dual basis from Example \ref{ex:dualofnonadjacentonn=6}, we obtain
\[
	C^{\sigma} \equiv F_{\{1,4,23,56\}}+F_{\{1,5,6,234\}}
	+F_{\{2,3,4,156\}}
\]
\[
	+F_{\{2,5,16,34\}}+F_{\{1,2,6,345\}}
	+F_{\{3,4,5,126\}}+F_{\{3,6,12,45\}}.
\]
\end{example}

\begin{example}
Let $n = 7$ and $\sigma = (123)(456) \in S_{7}$. Then $C^{\sigma}$ intersects 
\[
D_{\{1,2,3\}}, D_{\{4,5,6\}}, D_{\{1,4\}}, D_{\{1,5\}}, D_{\{1,6\}}, D_{\{2,4\}}, D_{\{2,5\}}, D_{\{2,6\}}, D_{\{3,4\}}, D_{\{3,5\}}, D_{\{3,6\}}
\] 
and all intersection  numbers are one. Among them, only $D_{\{1,2,3\}}, D_{\{4,5,6\}}, D_{\{3,4\}}$ are adjacent divisors. Thus 
\[
	C^{\sigma} \equiv F_{\{1,4,23,567\}}+F_{\{1,5,67,234\}}+
	F_{\{1,6,7,2345\}}
\]
\[
	+F_{\{2,3,4,1567\}}+F_{\{2,5,34,167\}}+F_{\{2,6,17,345\}}+
	F_{\{3,4,5,1267\}}+F_{\{3,6,45,127\}}.
\]
\end{example}

\begin{remark}
The numerical class of $\Mzn^G$ often has more symmetry beyond the
input group $G \subset S_n$.
For instance, for the curve above, we computed that the stabilizer of
$[C^{\sigma}]$ has order 72.  See the code samples on the second
author's website for more details.  
\end{remark}
\subsection{Cone of F-curves and $C^{\sigma}$}

We can ask several natural questions about the curve classes
$C^{\sigma}$ which have implications for the birational geometry of
$\Mzn$. Is $C^{\sigma}$ movable? Or as another extremal case, is
$C^{\sigma}$ rigid? A movable curve can be used to compute a facet of
the effective cone. On the other hand, a rigid rational curve on $\Mzn$
could be a candidate for a possible counterexample to the F-conjecture
(\cite[Section 2]{CT12}). 

In this section, we show that $C^{\sigma}$ is always movable and it is
on the boundary of the Mori cone $\overline{\mathrm{NE}}_{1}(\Mzn)$. 

\begin{lemma}\label{lem:degreeofprojection}
Let $\sigma \in S_{n}$ be a balanced permutation with two nontrivial
disjoint cycles. Pick four indexes $S := \{i,j,k,l\} \subset [n]$ and
let $\pi : \Mzn \to \ovop{M}_{0,S} \cong \PP^{1}$ be the
forgetful map. If $S \cap \sigma_{1} = \emptyset$ or
$S\cap\sigma_{2}=\emptyset$, then $\deg \pi_{*}[C^{\sigma}] = 0$.
\end{lemma}

\begin{proof}
Note that for a forgetful map $\pi : \Mzn \to \ovop{M}_{0,
  \{i,j,k,l\}}$, the locus of smooth curves maps to the locus of
smooth curves. It is straightforward to check that if $S\cap
\sigma_{1} =\emptyset$ or $S \cap \sigma_{2} = \emptyset$, then each
image of $(X, x_{1}, x_{2}, \cdots, x_{n}) \in C^{\sigma}$ is a smooth
curve in $\ovop{M}_{0, \{i,j,k,l\}}$. Therefore $ C^{\sigma}
\rightarrow \Mzn \rightarrow \ovop{M}_{0,S} $ is not surjective and $\deg \pi_{*}[C^{\sigma}] = 0$.
\end{proof}

\begin{proposition}\label{prop:ontheboundary}
For $n \ge 7$, all $C^{\sigma}$ are on the boundary of the Mori cone
$\overline{\mathrm{NE}}_{1}(\Mzn)$.
\end{proposition}

\begin{proof}
Suppose that $\sigma = \sigma_{1}\sigma_{2}$. Then one of $\sigma_{1}$, $[n] - \sigma_{1}$ has four or more elements. If you take four of them and denote the set of them by $S = \{i,j,k,l\}$, then by Lemma \ref{lem:degreeofprojection}, $\pi_{*}[C^{\sigma}] = 0$ for the projection $\pi : \Mzn \to \ovop{M}_{0,S}$. 

Now assume that there is an effective linear combination
\[
	C^{\sigma} \equiv \sum_{I}x_{I}E_{I}
\]
for some effective curves $E_{I}$ and $x_{I} \in \QQ_{+}$. For the projection $\pi: \Mzn \to \ovop{M}_{0,S}$, an F-curve maps to a point if at least two of $i, j, k, l$ are on the same partition. Also an F-curve maps to isomorphically onto $\ovop{M}_{0,S}$ if none of $i, j, k, l$ are on the same partition. Thus if we take an F-curve $F_{J}$ with a partition splits $S$ into four singleton sets, then $\pi_{*}[F_{J}] = [\ovop{M}_{0,S}]$. Therefore $x_{J}$ must be zero since $\pi_{*}[C^{\sigma}] = 0$. In other words, for \emph{any} effective linear combination of effective-curves, $F_{J}$ does not appear. This implies that $C^{\sigma}$ is on a facet which is disjoint from the ray generated by $F_{J}$. Therefore $C^{\sigma}$ is on the boundary of $\overline{\mathrm{NE}}_{1}(\Mzn)$.
\end{proof}

\begin{remark}
By using the same argument, we can show that for $\sigma = (12)(34) \in S_{6}$, $C^{\sigma}$ is on the boundary of $\overline{\mathrm{NE}}_{1}(\ovop{M}_{0,6})$. 
\end{remark}

\begin{theorem}\label{thm:cycliccasemovable}
For $n \ge 7$, all $C^{\sigma}$ are movable. 
\end{theorem}

\begin{proof}
Let $\sigma = \sigma_{1}\sigma_{2}$ and $G = \langle\sigma \rangle$. Also Suppose that the length of $\sigma_{i}$ is $r$. By Lemma \ref{lem:degreeofprojection}, for two projections $\pi_{1}:\Mzn \to \ovop{M}_{0,[n]-\sigma_{1}}$ and $\pi_{2}:\Mzn \to \ovop{M}_{0,[n]-\sigma_{2}}$, $C^{\sigma}$ is contained in a fiber. Moreover, $\pi : \Mzn \to \ovop{M}_{0,[n]-\sigma_{1}} \times \ovop{M}_{0,[n]-\sigma_{2}}$ is surjective, because $|([n]-\sigma_{1})\cap ([n]-\sigma_{2})| = j \le 2$. Note that the dimension of a general fiber is $(n-3)-2(n-r-3) = 2r - n+3 = 3-j$. 

On the interior of $\ovop{M}_{0,[n]-\sigma_{1}}$, there exists a unique point $p$ which is invariant with respect to $\langle \sigma_{2}\rangle$-action. Indeed, for a $\PP^{1}$ with specific coordinates, choose two pivotal points determining the rotational axis, and another point for one of marked points in $\sigma_{2}$. Then all other marked points for the curve parametrized by $p$ are determined by the $\sigma_{2}$-action. Thus there is a three dimensional moduli and if we apply $\mathrm{PGL}_{2}$-action, then we obtain a point. Similarly, we have a $\langle \sigma_{1}\rangle$-invariant point $q$ on $\ovop{M}_{0,[n]-\sigma_{2}}$. If we regard the induced $G = \langle\sigma\rangle$-action on $\ovop{M}_{0,[n]-\sigma_{1}}\times \ovop{M}_{0, [n]-\sigma_{2}}$, then $\pi : \Mzn \to \ovop{M}_{0,[n]-\sigma_{1}}\times \ovop{M}_{0, [n]-\sigma_{2}}$ is $G$-equivariant and $C^{\sigma}$ should be in the fiber $\pi^{-1}(p, q)$.

If $j = 2$, a general fiber is a curve. And the special fiber $\pi^{-1}(p, q)$ is an irreducible curve, because a general point of it is determined by the cross ratio of four marked points $x_{1}, x_{2}, x_{3}, x_{4}$ where $x_{1} \in \sigma_{1}$, $x_{2} \in \sigma_{2}$ and $x_{3}, x_{4}$ are two fixed points. So $C^{\sigma} = \pi^{-1}(p, q)$. Therefore $C^{\sigma}$ is movable. 

Finally, if $j < 2$, then $C^{\sigma}$ is the image of $C^{\sigma'}$ for $\rho : \ovop{M}_{0,n+2-j} \to \Mzn$ by Lemma \ref{lem:Crj-1andCrj}. Note that $C^{\sigma'}$ is in an algebraic family $\mathcal{C}$ covers $\ovop{M}_{0,n+2-j}$, by above observation. By composing with $\rho$, we obtain a family $\mathcal{C'}$ containing $C^{\sigma}$ and covers $\Mzn$. Thus $C^{\sigma}$ is movable, too.
\end{proof}

\section{Dihedral group cases}\label{sec:dihedral}

Next we will discuss the case where $G$ is a dihedral
group. Surprisingly, the geometry of $\Mzn^{G}$ when $G$ is a dihedral
group is very different from the geometry of $\Mzn^{G}$ when $G$ is a
cyclic group.

Let $G$ be a subgroup of $S_{n}$ which is isomorphic to $D_{k}$ with
$k \ge 3$. Then, up to conjugation, $G$ acts on $\PP^{1} \cong S^{2}$,
as the symmetry group of a bipyramid over a regular $k$-gon. There is
a unique orbit (corresponding to two pivotal points) of order two,
there are two orbits of order $k$, and all other orbits have order $2k$. 

Since $G$ permutes marked points, if a point $x$ on $\PP^{1}$ is a marked point, then all points in $G\cdot x$ are marked points, too. So if we have $(\PP^{1}, x_{1}, x_{2}, \cdots, x_{n}) \in \Mzn^{G}$, we have a partition of 
$[n]$ into subsets of size $2$, $k$, and $2k$. The number of size $2$ orbits is 0 or 1. The number of size $k$ orbits is at most two. The dimension of the stratum is the number of orbits of order $2k$ because to fix a $D_{k}$-action on $\PP^{1}$, we need to pick three points already, which are two pivotal points and a point of order $k$. 

Define $G$-invariant families of $n$-pointed $\PP^{1}$ as
following. Fix a coordinate on $\PP^{1}$ and fix a $G$-action on
$\PP^{1}\cong S^{2}$ as a rotation group of a bipyramid over a regular
$k$-gon. Take a generic orbit $O$ of order $2k$,  specify the number
of \emph{special orbits} (orbits of order 2 and $k$), and assign 
marked points to some of these orbits. Then we have an element of
$\op{M}_{0,n}$. By varying $O$, we obtain a rational map $f: \PP^{1}
\dashrightarrow \Mzn$. Since $\PP^{1}$ is a curve, this rational map
can be extended to all of $\PP^{1}$. $\Mzn^{G}$ is the image of $f$.

\begin{remark}
If there are no special marked orbits, i.e., if we choose a general orbit only, then the one-dimensional family over $\PP^{1}$ obtained by varying the orbit of order $2k$ is a $(2:1)$ cover of a rational curve $\Mzn^{G}$ on $\Mzn$.
\end{remark}

In summary, for a dihedral group $G \cong D_{k}$, $\Mzn^{G}$ is a
curve if there exists a unique $G$-orbit of order $2k$. There might be
an extra marked orbit of order 2 and at most two marked orbits of order $k$. By an
argument similar to that of  Lemma \ref{lem:isomorphictoP1}, we obtain that $\Mzn^{G}$ is isomorphic to $\PP^{1}$. Thus we get the following result. 

\begin{lemma} \label{lem:dihedralcurvesnk}
Let $G \cong D_{k}$. $\Mzn^{G}$ is a curve on $\Mzn$ only if $n = 2k, 2k+2, 3k, 3k+2, 4k$, and $4k+2$.\end{lemma}

\begin{definition}
For a dihedral group $G \cong D_{k} \subset S_{n}$, suppose that
$\Mzn^{G}$ is a curve. We say that the $G$ action is \emph{of type $(a,b)$} if $a$ is the number of order two marked orbits and $b$ is the number of order $k$ special orbits. So $0 \le a \le 1$ and $0 \le b \le 2$.
\end{definition}

If there is no orbit of order $k$ consisting of marked points, then $\Mzn^{G}$ is a curve we have already described:

\begin{lemma} \label{lem:whendihedraliscyclic}
Let $G \cong D_{k}$ such that $\Mzn^{G}$ is a curve. Let $\sigma \in
G$ have order $k$. If the $G$ action is of type $(a, 0)$, then $\Mzn^{G} = C^{\sigma}$. 
\end{lemma} 

\begin{proof}
It is clear that $\Mzn^{G} \subset \Mzn^{\langle \sigma\rangle} =
C^{\sigma}$, and both of them are irreducible curves, so $\Mzn^{G} = C^{\sigma}$. 
\end{proof}

\begin{example}\label{ex:k=3C_00}
The simplest case is $G \cong D_{3}$ and $n = 6$. Let $G = \langle (123)(456), (14)(26)(35)\rangle$. Then $G$ is of type $(0, 0)$ and it is the rotation group of a triangular prism inscribed in $S^{2}$ whose top vertices are $1,2,3$ and whose bottom vertices are $4,5,6$ in the same order. 
\end{example}

To compute the numerical class of $\Mzn^{G}$, we need to compute the intersection numbers with boundary divisors. A point configuration on a general point of $\Mzn^{G}$ degenerates if the `moving' orbit of order $2k$ approaches a special orbit (an orbit of order $2$ or $k$). Note that special orbits might not consist of marked points. 

\begin{proposition}\label{prop:intersectiondihedral}
Let $G \subset S_{n}$ be a finite group isomorphic to $D_{k}$ for $k \ge 3$. Suppose that $\Mzn^{G}$ is a curve. Let $\sigma \in G$ be an order $k$ element.
\begin{enumerate}
\item If the $G$ action is of type $(a,0)$, then $\sigma$ is a product of two disjoint cycles $\sigma_{1}$ and $\sigma_{2}$ of length $k$. 
\begin{enumerate}
\item If $a = 0$, 
\[
	D_{I} \cdot \Mzn^{G} = \begin{cases}2, & I = \sigma_{1} \mbox{ or }
	I = \sigma_{2},\\
	1, & I = \{i, j\} \mbox{ where } i \in \sigma_{1}, j \in \sigma_{2},\\
	0, & \mbox{otherwise.}\end{cases}
\]
\item If $a = 1$, let $\ell$ be one of two marked points in the orbit of order two. Then:
\[
	D_{I} \cdot \Mzn^{G} = \begin{cases}
	1, & I = \sigma_{1} \cup \{\ell\} \mbox{ or } \sigma_{2} \cup \{\ell\},\\
	1, & I = \{i,j\}, i \in \sigma_{1}, j \in \sigma_{2},\\
	0, & \mbox{otherwise.}
	\end{cases}
\]
\end{enumerate}
\item If the $G$ action is of type $(a,1)$, then $\sigma$ is a product of three disjoint cycles $\sigma_{1}, \sigma_{2}, \sigma_{3}$ of order $k$. If we take a reflection $\tau \in G - \langle \sigma \rangle$, then exactly one of them (say $\sigma_{3}$) is invariant for $\tau$-action. Furthermore, we are able to take $\tau$, such that there is $m \in \sigma_{3}$ so that $\tau(m) = m$. 
\begin{enumerate}
\item If $a = 0$ and $k$ is odd, then:
\[
	D_{I} \cdot \Mzn^{G} = \begin{cases}2, & I = \sigma_{1} \mbox{ or } 
	I = \sigma_{2},\\
	1, & I = \{i, j\} \mbox{ where } i \in \sigma_{1}, j \in \sigma_{2},\\
	1, & I = \{\sigma^{t}(i), \sigma^{t}(\tau(i)), \sigma^{t}(m)\} \mbox{ where } 
	i \in \sigma_{1}, 0 \le t < k,\\
	0, & \mbox{otherwise.}
	\end{cases}
\]
\item If $a = 0$ and $k$ is even, then:
\[
	D_{I} \cdot \Mzn^{G} = \begin{cases}2, & I = \sigma_{1} \mbox{ or } 
	I = \sigma_{2},\\
	2, & I = \{i, \sigma^{2t+1}(\tau(i))\} \mbox{ where } i \in \sigma_{1}, 
	0 \le t < k/2,\\
	1, & I = \{\sigma^{t}(i), \sigma^{t}(\tau(i)), \sigma^{t}(m)\} \mbox{ where } 
	i \in \sigma_{1}, 0 \le t < k,\\
	0, & \mbox{otherwise.}
	\end{cases}
\]
\item If $a = 1$ and $k$ is odd, pick $\ell$ with order two. Then we have:
\[
	D_{I} \cdot \Mzn^{G} = \begin{cases}1, & I = \sigma_{1} \cup \{\ell\} \mbox{ or }
	\sigma_{2} \cup \{\ell\},\\
	1, & I = \{i, j\} \mbox{ where } i \in \sigma_{1}, j \in \sigma_{2},\\
	1, & I = \{\sigma^{t}(i), \sigma^{t}(\tau(i)), \sigma^{t}(m)\}, i \in \sigma_{1}, 
	0 \le t < k,\\
	0, & \mbox{otherwise.}
	\end{cases}
\]
\item If $a = 1$ and $k$ is even, pick $\ell$ with order two. Then we have:
\[
	D_{I} \cdot \Mzn^{G} = \begin{cases}1, & I = \sigma_{1} \cup \{\ell\} \mbox{ or }
	\sigma_{2} \cup \{\ell\},\\
	2, & I = \{i, \sigma^{2t+1}(\tau(i))\} \mbox{ where } i \in \sigma_{1}, 
	0 \le t < k/2,\\
	1, & I = \{\sigma^{t}(i), \sigma^{t}(\tau(i)), \sigma^{t}(m)\}, i \in \sigma_{1}, 
	0 \le t < k,\\
	0, & \mbox{otherwise.}
	\end{cases}
\]
\end{enumerate}
\item If the $G$ action is of type $(a,2)$, then $\sigma$ is a product of four disjoint cycles $\sigma_{a}$, $1 \le a \le 4$ of order $k$. For any reflection $\tau \in G - \langle \sigma \rangle$, two of $\sigma_{a}$'s (say $\sigma_{3}$ and $\sigma_{4}$) are $\tau$-invariant. Furthermore, by taking appropriate reflections $\tau_{1}, \tau_{2}$, we may assume that there is $m \in \sigma_{3}$ and $n \in \sigma_{4}$ such that $\tau_{1}(m)  = m$ and $\tau_{2}(n) = n$. 
\begin{enumerate}
\item If $a = 0$, 
\[
	D_{I} \cdot \Mzn^{G} = \begin{cases}2, & I = \sigma_{1} \mbox{ or } 
	I = \sigma_{2},\\
	1, & I = \{\sigma^{t}(i), \sigma^{t}(\tau_{1}(i)), \sigma^{t}(m)\} \mbox{ where }
	i \in \sigma_{1}, 0 \le t < k,\\
	1, & I = \{\sigma^{t}(i), \sigma^{t}(\tau_{2}(i)), \sigma^{t}(n)\} \mbox{ where }
	i \in \sigma_{1}, 0 \le t < k,\\
	0, & \mbox{otherwise.}
	\end{cases}
\]
\item If $a = 1$ and if $\ell$ is one of two marked points of order two, 
\[
	D_{I} \cdot \Mzn^{G} = \begin{cases}1, & I = \sigma_{1} \cup \{\ell\}
	\mbox{ or } I = \sigma_{2} \cup \{\ell\},\\
	1, & I = \{\sigma^{t}(i), \sigma^{t}(\tau_{1}(i)), \sigma^{t}(m)\} \mbox{ where }
	i \in \sigma_{1}, 0 \le t < k,\\
	1, & I = \{\sigma^{t}(i), \sigma^{t}(\tau_{2}(i)), \sigma^{t}(n)\} \mbox{ where }
	i \in \sigma_{1}, 0 \le t < k,\\
	0, & \mbox{otherwise.}
	\end{cases}
\]
\end{enumerate}
\end{enumerate}
\end{proposition}

\begin{proof}
Item (1) is just a restatement of Proposition
\ref{prop:intersectionboundary}. Items (2) and (3) can be obtained by
considering when a general orbit $\sigma_{1} \cup \sigma_{2}$ can
collide with a point in special orbits ($\sigma_{3}, \sigma_{4},
\{\ell, \tau(\ell)\}$). We leave the proof as an exercise for the reader.
\end{proof}

\begin{example}\label{ex:dihedralcurve}
Let $G = \langle (123)(456)(789), (14)(26)(35)(89)\rangle \cong D_{3}$
and $C = \ovop{M}_{0,9}^{G}$ on $\ovop{M}_{0,9}$. In this case, $G$ is
of type $(0, 1)$. Then the marked points $x_{1}, \cdots, x_{6}$
form an orbit of order $6$ and the marked points $x_7, x_8, x_9$
form an orbit of order $3$. $C \cdot D_{\{1,2,3\}} = C \cdot D_{\{4,5,6\}} = 2$ and $C \cdot D_{\{i,j\}} = 1$ when $1 \le i \le 3$ and $4 \le j \le 6$. Also $C \cdot D_{I} = 1$ if $D_{I}$ is one of
\[
	D_{\{1,4,7\}}, D_{\{2,5,8\}}, D_{\{3,6,9\}}, 
	D_{\{1,6,8\}}, D_{\{2,4,9\}}, D_{\{3,5,7\}},
	D_{\{1,5,9\}}, D_{\{2,6,7\}}, D_{\{3,4,8\}}.
\]
All other intersection numbers are zero.

An interesting fact about $C$ is that for \emph{every} projection
$\pi: \ovop{M}_{0,9} \to \ovop{M}_{0,4}$, $C$ is not contracted,
because for any four of nine points, their cross ratio is not a
constant. In particular, $C$ is not a fiber of a hypergraph morphism (\cite[Definition 4.4]{CT12}), thus it is not a curve constructed in \cite{CT12}.

Proposition \ref{prop:intersectiondihedral} allows us
to write the class of $C$ as a (non-effective) linear combination of
F-curves.  We create a vector of intersection numbers with the
nonadjacent basis of divisors, and use the command
\texttt{writeCurveInSingletonSpineBasis} in the \texttt{M0nbar}
package for \texttt{Macaulay2} to obtain the coefficients of $C$ in
the so-called ``singleton spine basis'' of F-curves.  The resulting
expression is supported on 37 F-curves.  26 curves in this expression
have positive coefficients, and 11 curves in this expression have
negative coefficients.

In Example \ref{ex:linearcombinationn=9} we will express the class of this
curve as an effective linear combination of F curves.
\end{example}

\begin{example}\label{ex:dihedraln=12}
For notational simplicity, write $[12] = \{1,2,3,4,5,6,7,8,9,a,b,c\}$. Let 
\[
	G = \langle (123)(456)(789)(abc), (14)(26)(35)(89)(ac)\rangle \cong D_{3}
\] 
and $C := \ovop{M}_{0,12}^{G}$ on $\ovop{M}_{0,12}$. Then $x_{1},
x_{2}, \cdots, x_{6}$ form an orbit of order $6$; $x_{7}, x_{8},
x_{9}$ form an orbit of order $3$; and $x_a, x_b, x_c$ form  an
orbit of order 3.  $D_{\{1,2,3\}} \cdot C = D_{\{4,5,6\}} \cdot C = 2$
and $D_{I} \cdot C = 1$ for the following 18 irreducible components of $D_{3}$, 
\[
	D_{\{1,4,7\}}, D_{\{2,5,8\}}, D_{\{3,6,9\}}, D_{\{1,6,8\}}, D_{\{2,4,9\}}, D_{\{3,5,7\}},
\]
\[
	D_{\{1,5,9\}}, D_{\{2,6,7\}}, D_{\{3,4,8\}}, D_{\{1,5,c\}}, D_{\{2,6,a\}}, D_{\{3,4,b\}}, 
\]
\[
	D_{\{1,4,a\}}, D_{\{2,5,b\}}, D_{\{3,6,c\}}, D_{\{1,6,b\}}, D_{\{2,4,c\}}, D_{\{3,5,a\}}.
\]

Proposition \ref{prop:intersectiondihedral} allows us
to write the class of $C$ as a (non-effective) linear combination of
F-curves.  We create a vector of intersection numbers with the
nonadjacent basis of divisors, and use the command
\texttt{writeCurveInSingletonSpineBasis} in the \texttt{M0nbar}
package for \texttt{Macaulay2} to obtain the coefficients of $C$ in
the so-called ``singleton spine basis'' of F-curves.  The resulting
expression is supported on 103 F-curves.  69 curves in this expression
have positive coefficients, and 34 curves in this expression have
negative coefficients.

In Section \ref{sec:example} we will express the class of this
curve as an effective linear combination of F-curves.  
\end{example}

\subsection{Rigidity}
For a rational curve $f : \PP^{1} \to X$ to a smooth projective variety $X$ of dimension $d$, 
\begin{equation}\label{eqn:deformationspace}
\dim_{[f]}\mathrm{Hom}(\PP^{1}, X) \ge -K_{X}\cdot f_{*}\PP^{1} + d\chi(\cO_{\PP^{1}}) = -K_{X}\cdot f_{*}\PP^{1} + d
\end{equation}
by \cite[Theorem 1.2]{Kol96}. If $f$ is rigid and $X = \Mzn$, then $\dim_{[f]}\mathrm{Hom}(\PP^{1}, \Mzn) \le \mathrm{Aut}(\PP^{1}) = 3$, thus $K_{\Mzn} \cdot f_{*}\PP^{1} \ge n - 6$. In particular, for $n \ge 7$, the intersection must be positive. 

On the other hand, 
\begin{equation}
K_{\Mzn} = \sum_{k=2}^{\lfloor \frac{n}{2}\rfloor}\left(-2 + \frac{k(n-k)}{n-1}\right)D_{k}
\end{equation}
so for $n \ge 7$, except the coefficient of $D_{2}$, all other coefficients are nonnegative. Thus if we want to find an example of a rigid curve, then its intersection with $D_{k}$ for $k \ge 3$ should be large compared to the intersection with $D_{2}$. 

Let $G \subset S_{n}$ be isomorphic to $D_{k}$ with odd $k$ and $n = 4k$. Example \ref{ex:dihedraln=12} is the case of $k = 3$. In this case, $\Mzn^{G} \cdot D_{3} = 2k^{2}$, $\Mzn^{G}\cdot D_{k} = 4$ and $\Mzn^{G}\cdot D_{i} = 0$ for all $i \ne 3, k$ by Proposition \ref{prop:intersectiondihedral}. So 
\[
	\Mzn^{G} \cdot K_{\Mzn} = 2k^{2}-8.
\]
Therefore for $k \ge 3$, it is a large positive number, and we may have a rigid curve. When $k = 3$, it gives a curve on $\ovop{M}_{0,12}$ (Example \ref{ex:dihedraln=12}). This example  is different from the rigid curve of \cite[Section 4]{CT12}, because it does not intersect $D_{2}$ or $D_{4}$. 

To show equality in \eqref{eqn:deformationspace}, we would need to
evaluate the normal bundle to the rational curve. The computation of the blow-up formula for the normal bundle is not easy unless the blow-up center is contained in the curve.

\section{Toric degenerations on Losev-Manin spaces}\label{sec:LosevManin}

Let $C = \Mzn^{G}$ be a curve of the type described in the previous
section.  We would like to compute the numerical class of a curve
$C$ and find an effective linear combination of F-curves which is numerically
equivalent to $C$. In the case of a dihedral group, the idea of using
the basis dual to the nonadjacent basis in Section
\ref{sec:efflincombcyclic} does not work anymore. To do this computation, instead of computing the class directly, we will use Losev-Manin space (\cite{LM00}) to find an approximation of it first. 

\subsection{Background on Losev-Manin spaces}

In this section, we give a brief review on Losev-Manin spaces. 

Fix a sequence of positive rational numbers $A = (a_{1}, a_{2}, \cdots, a_{n})$ where $0 < a_{i} \le 1$ and $\sum a_{i} > 2$. Then the moduli space $\Mza$ of weighted pointed stable curves (\cite{Has03}) is the moduli space of pairs $(X, x_{1}, x_{2}, \cdots, x_{n})$ such that:
\begin{itemize}
	\item $X$ is a connected, reduced projective curve of $p_{a}(X) = 0$,
	\item $(X, \sum a_{i}x_{i})$ is a semi-log canonical pair, 
	\item $\omega_{X}+\sum a_{i}x_{i}$ is ample. 
\end{itemize}
It is well known that $\Mza$ is a fine moduli space of such pairs, and there exists a divisorial contraction $\rho : \Mzn \to \Mza$ which preserves the interior $\mathrm{M}_{0,n}$. For more details, see \cite{Has03}. 

\begin{definition}\cite{LM00}, \cite[Section 6.4]{Has03}
The \emph{Losev-Manin space} $\Ln$ is a special case of Hassett's weighted $(n+2)$-pointed stable curves such that two of weights are $1$ and the rest of them are $\frac{1}{n}$. For example, if the first and second points are weight $1$ points, then $\Ln = \ovop{M}_{0,\left(1,1,\frac{1}{n}, \frac{1}{n}, \cdots, \frac{1}{n}\right)}$.
\end{definition}

The Losev-Manin space $\Ln$ is a
projective toric variety, and among toric $\Mza$, it is the closest one
to $\ovop{M}_{0,n+2}$. Indeed, $\Ln$ is a toric variety whose
corresponding polytope is the permutohedron of dimension $n-1$
(\cite[Section 7.3]{GKZ08}), which is obtained from an ($n-1$)-dimensional simplex by
carving smaller dimensional faces. So $\Ln$ can be obtained by
successive blow-ups of a projective space as following. Let $p_{1},
p_{2}, \cdots, p_{n}$ be the standard coordinate points of
$\PP^{n-1}$. For a nonempty subset $I \subset [n]$, let $L_{I}$ be the
linear subspace of $\PP^{n-1}$ spanned by $\{p_{i}\}_{i \in
  I}$. Blow-up $\PP^{n-1}$ along the $n$ coordinates points.  Next, blow-up along the proper transforms of $L_{I}$ with $|I| = 2$. After that, blow-up along the proper transforms of $L_{I}$ with $|I| = 3$. If we perform these blow-ups along all $L_{I}$ up to $|I| = n-3$, we obtain $\Ln$. 

Note that there is a dominant reduction morphism $\rho : \ovop{M}_{0,n+2} \to \Ln$ (\cite[Theorem 4.1]{Has03}), because it is a special case of Hassett's space. If $i$ and $j$-th points are weight $1$ points on $\Ln$, we will use the notation $\rho_{i,j}$ for the reduction map $\rho$, to indicate which points are weight 1 points on $\Ln$.

As a moduli space, $\Ln$ is a fine moduli space of chains of rational curves. With this weight distribution, all stable rational curves are chains of $\PP^{1}$. Moreover, if $p$, $q$ are two points with weight 1 and $x_{1}, x_{2}, \cdots, x_{n}$ are points with weight $\frac{1}{n}$, then one of $p$, $q$ is on one of end components and the other one is on the other end component. We can pick one of them (say $p$) as the $0$-point. The other point becomes the $\infty$-point. (This notation will be justified soon.) So each boundary stratum corresponds to an \emph{ordered} partition of $[n] := \{1,2, \cdots, n\}$ by reading the subset of marked points on each irreducible component (from $0$-point to $\infty$-point). For example, the trivial partition $[n]$ corresponds to the big cell. The partition $I | J$ corresponds a toric divisor $D_{I \cup \{0\}} = D_{J \cup \{\infty\}}$. A partition $I_{1}|I_{2}|\cdots |I_{n-1}$ of length $n-1$ (so only one of $I_{i}$ has two elements and the others are singleton sets) corresponds to a toric boundary curve. 

\begin{definition}
For an ordered partition $I_{1}|I_{2}|\cdots|I_{k}$ of $[n]$, we say a rational chain 
\[
	(X, x_{1}, x_{2}, \cdots, x_{n}, 0, \infty) \in \Ln
\]
is \emph{of type $I_{1}|I_{2}|\cdots|I_{k}$} if 
\begin{itemize}
\item $X$ is a rational chain of $k$ projective lines $X_{1}, X_{2}, \cdots, X_{k}$;
\item $0 \in X_{1}$, $\infty \in X_{k}$;
\item $x_{i} \in X_{j}$ if and only if $i \in I_{j}$. 
\end{itemize}
\end{definition}

\begin{figure}[!ht]
\begin{tikzpicture}[scale=0.4]
	\draw[line width=1pt] (0, 5) -- (6, 0);
	\draw[line width=1pt] (5, 0) -- (11, 5);
	\draw[line width=1pt] (10, 5) -- (16, 0);
	\draw[line width=1pt] (15, 0) -- (21, 5);
	\fill (1.2, 4) circle (4pt);
	\fill (19.8, 4) circle (4pt);
	\fill (3.6, 2) circle (4pt);
	\fill (7.4, 2) circle (4pt);
	\fill (9.8, 4) circle (4pt);
	\fill (12.4, 3) circle (4pt);
	\fill (14.8, 1) circle (4pt);
	\fill (17.4, 2) circle (4pt);
	\fill (18.6, 3) circle (4pt);
	
	\node (0) at (0.2, 4) {$0$};
	\node (inf) at (20.8, 4) {$\infty$};
	\node (1) at (2.6, 2) {$1$};
	\node (4) at (6.4, 2) {$4$};
	\node (3) at (8.8, 4) {$3$};
	\node (2) at (11.4, 3) {$2$};
	\node (5) at (13.8, 1) {$5$};
	\node (7) at (18.4, 2) {$7$};
	\node (6) at (19.6, 3) {$6$};	
\end{tikzpicture}
\caption{A rational chain of type $1|34|25|67$}
\end{figure}
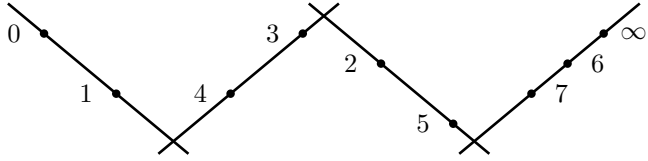

For each marked point $x_{i}$ with weight $1/n$, we are able to take a
one parameter subgroup $T_{i}$ of $(\CC^{*})^{n-1} \subset \Ln$, which
moves $x_{i}$ only. $T_{i}$ acts on $(X, x_{1}, x_{2}, \cdots, x_{n},
0, \infty)$ as a multiplication of $\CC^{*}$ on the component
containing $x_{i}$. Note that every component of $X$ has two special
points (singular points or points with weight 1). Let $X_{i}$ be the
irreducible component containing $x_{i}$, and let $y$ be the special point of $X_{i}$ which is closer to $0$-point. For $t \in T_{i}$, the limit 
\[
	\lim_{t \to 0}t \cdot (X, x_{1}, x_{2}, \cdots, x_{n}, 0, \infty)
\]
is the $(\frac{1}{n}, \frac{1}{n}, \cdots, \frac{1}{n}, 1, 1)$-stable
curve obtained by first making a bubble at $y$, then putting $x_{i}$
on the bubble, and then stabilizing. If $x_{i}$ is the only marked point with weight $1/n$ on $X_{i}$, then $(X, x_{1}, \cdots, x_{n}, 0, \infty)$ is $T_{i}$-invariant. 

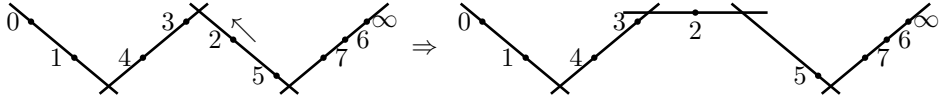
\begin{figure}[!ht]
\begin{tikzpicture}[scale=0.24]
	\draw[line width=1pt] (0, 5) -- (6, 0);
	\draw[line width=1pt] (5, 0) -- (11, 5);
	\draw[line width=1pt] (10, 5) -- (16, 0);
	\draw[line width=1pt] (15, 0) -- (21, 5);
	\fill (1.2, 4) circle (5pt);
	\fill (19.8, 4) circle (5pt);
	\fill (3.6, 2) circle (5pt);
	\fill (7.4, 2) circle (5pt);
	\fill (9.8, 4) circle (5pt);
	\fill (12.4, 3) circle (5pt);
	\fill (14.8, 1) circle (5pt);
	\fill (17.4, 2) circle (5pt);
	\fill (18.6, 3) circle (5pt);
	
	\node (0) at (0.2, 4) {$0$};
	\node (inf) at (20.8, 4) {$\infty$};
	\node (1) at (2.6, 2) {$1$};
	\node (4) at (6.4, 2) {$4$};
	\node (3) at (8.8, 4) {$3$};
	\node (2) at (11.4, 3) {$2$};
	\node (5) at (13.8, 1) {$5$};
	\node (7) at (18.4, 2) {$7$};
	\node (6) at (19.6, 3) {$6$};
	
	\node (arrow) at (13, 3.5) {$\nwarrow$};

	\node (to) at (23, 2.5) {$\Rightarrow$};
	
	\draw[line width=1pt] (25, 5) -- (31, 0);
	\draw[line width=1pt] (30, 0) -- (36, 5);
	\draw[line width=1pt] (34, 4.5) -- (42, 4.5);
	\draw[line width=1pt] (40, 5) -- (46, 0);
	\draw[line width=1pt] (45, 0) -- (51, 5);
	\fill (26.2, 4) circle (5pt);
	\fill (49.8, 4) circle (5pt);
	\fill (28.6, 2) circle (5pt);
	\fill (32.4, 2) circle (5pt);
	\fill (34.8, 4) circle (5pt);
	\fill (38, 4.5) circle (5pt);
	\fill (44.8, 1) circle (5pt);
	\fill (47.4, 2) circle (5pt);
	\fill (48.6, 3) circle (5pt);
	
	\node (00) at (25.2, 4) {$0$};
	\node (inff) at (50.8, 4) {$\infty$};
	\node (11) at (27.6, 2) {$1$};
	\node (44) at (31.4, 2) {$4$};
	\node (33) at (33.8, 4) {$3$};
	\node (22) at (38, 3.5) {$2$};
	\node (55) at (43.8, 1) {$5$};
	\node (77) at (48.4, 2) {$7$};
	\node (66) at (49.6, 3) {$6$};

\end{tikzpicture}
\caption{A description of $\displaystyle\lim_{t \to 0}t \cdot (X, x_{1}, x_{2}, \cdots, x_{n}, 0, \infty)$ for $t \in T_{2}$}
\end{figure}

On $\Ln$, torus invariant divisors are all the images of boundary divisors on $\ovop{M}_{0,n+2}$ of the form $D_{I}$ where $0 \in I$ and $\infty \notin I$. Thus all toric boundary cycles (which are intersections of toric divisors) are images of F-strata. In particular, all 1-dimensional toric boundary cycles are images of F-curves.

\subsection{Computing limit cycles}

In this section, we describe a method to compute a numerically
equivalent effective linear combination of toric boundary curves for a
given effective curve $\overline{C} \subset \Ln$. Because $\Ln$ is a
toric variety, the Mori cone is generated by torus invariant
curves. Thus, there exists an effective linear combination of toric
boundary cycles representing $[\overline{C}]$. We will apply this idea
to the image $\rho(C)$ for an invariant curve $C$ on
$\ovop{M}_{0,n+2}$. Thus this effective linear combination is an
approximation of F-curve linear combination of $C$ on
$\ovop{M}_{0,n+2}$. Later, in Section \ref{sec:efflincomb}, we will
discuss a strategy that uses this approximation to find an effective linear combination of F-curves for $C$ on $\ovop{M}_{0,n+2}$.

The basic idea of computing limit cycles on $\Ln$ is the
following. For each marked point $x_{i}$ with weight $1/n$, we have a
one parameter subgroup $T_{i} \subset (\CC^{*})^{n-1}$ which moves
$x_{i}$ only. So if we choose an ordering of the marked points $x_{1}, x_{2}, \cdots, x_{n}$, then we have a sequence of one parameter subgroups $T_{1}, T_{2}, \cdots, T_{n-1}$, where $T_{i}$ is the one parameter subgroup moving $x_{i}$. Note that they generate the big torus $T := (\CC^{*})^{n-1} \subset \Ln$. 

Let $\overline{C} \subset \Ln$ be an effective curve. Let $\CLn$ be
the Chow variety parameterizing algebraic cycles of dimension one. Let
$C^{0} := \overline{C}$ and consider $[C^{0}] \in \CLn$. Obviously $T
\subset \Ln$ acts on $\CLn$. We will compute the limit
$[C^{1}]:=\lim_{t \to 0}t\cdot [C^{0}]$ for $t \in T_{1}$. Then
$[C^{1}]$ is a $T_{1}$-invariant point on $\CLn$ and $[C^{1}]$ is an
effective cycle on $\Ln$. In Section \ref{ssec:limitcomponents} we
describe a way to find the irreducible components of $[C^{1}]$, and in
Section \ref{ssec:multiplicities} we explain how to compute the
multiplicity of each irreducible component. 

For the next step, we can compute $[C^{2}] := \lim_{t \to 0}t\cdot
[C^{1}]$ for $t \in T_{2}$, componentwise. Note that each irreducible
component of $[C^{1}]$ is contained in an irreducible component of
boundary, which is isomorphic to the product $\prod \ovop{L}_{k}$ for
small $k$. Also $T_{2}$ acts on exactly one of $\ovop{L}_{k}$. Hence
the limit computation is very similar to the previous computation. If we perform the limit computations successively for $T_{i}$ with $1 \le i \le n-1$, then the limit cycle is invariant for all $T_{i}$. Therefore it is invariant for $T$, and the limit $[C^{n-1}]$ is a linear combination 
\[
	[C^{n-1}] = \sum b_{i}[B_{i}]
\]
of torus invariant curves $B_{i}$. Then $[C^{n-1}]$ is numerically equivalent to $[C^{0}] = [\overline{C}]$. Moreover, each torus invariant curve $B_{i}$ is the image of a unique F-curve $F_{I_{i}}$ on $\ovop{M}_{0,n+2}$. 

We summarize this process in Section \ref{ssec:summary}.

\begin{remark}
The method described above for computing a numerically equivalent effective linear combination of boundaries works for any toric variety and any curve on it. Of course, if the given toric variety is complicated, the actual computation is hopeless. But in the case of $\Ln$, even though as a toric variety it is very complicated, this computation is doable because of its beautiful modular interpretation and inductive structure. 
\end{remark}

\subsection{Limit components}\label{ssec:limitcomponents}

In this section, we explain how to find irreducible components
appearing in the limit of given curve on $\Ln$. We will describe our
method with an example $\overline{C} = \rho_{8,9}(C)$, where $C$ is
the $D_{3}$-invariant curve on $\ovop{M}_{0,9}$ from Example
\ref{ex:k=3C_00}. We will use the reduction map $\rho_{8,9} :
\ovop{M}_{0,9} \to \ovop{L}_{7}$ and take the $8^{th}$ point as our
$0$-point and $9^{th}$ marked point as the $\infty$-point. 

A general point of $C \cap \mathrm{M}_{0,9}$ can be written as 
\[
\left(\PP^{1}, z, \omega z, \omega^{2}z, \frac{1}{z}, \frac{\omega}{z}, 	\frac{\omega^{2}}{z}, 1, \omega, \omega^{2}\right)
\]
where $z$ is a coordinate function and $\omega$ is a cubic root of unity. By using a M\"obius transform $x \mapsto \frac{1-\omega^{2}}{1-\omega}\cdot \frac{x-\omega}{x-\omega^{2}}$, we obtain new coordinates
\small
\[
C(z) := \left(\PP^{1}, \alpha \frac{z-\omega}{z-\omega^{2}}, \alpha \frac{\omega z - \omega}{\omega z - \omega^{2}}, \alpha \frac{\omega^{2}z - \omega}{\omega^{2}z-\omega^{2}}, \alpha\frac{1-\omega z}{1-\omega^{2} z}, \alpha \frac{\omega - \omega z}{\omega - \omega^{2}z}, \alpha \frac{\omega^{2}-\omega z}{\omega^{2}-\omega^{2} z}, 1, 0, \infty\right)
\]
\normalsize
where $\alpha = \frac{1-\omega^{2}}{1-\omega}$.

From this description of the general point of $C$, we can recover 
the special
points on $C$, which correspond to singular curves via stable
reduction. For example, $\lim_{z \to \omega}C(z)$ on $\ovop{M}_{0,9}$
is the point corresponding to the following rational curve with four irreducible components. 

\begin{figure}[!ht]
\begin{tikzpicture}[scale=0.5]
	\draw[line width=1pt] (0, 0) -- (10, 0);
	\draw[line width=1pt] (1, -1) -- (0, 5);
	\draw[line width=1pt] (5, -1) -- (5, 5);
	\draw[line width=1pt] (9, -1) -- (10, 5);
	\fill (0.1, 4) circle (4pt);
	\node (8) at (-0.8, 4) {$8$};
	\fill (9.8, 4) circle (4pt);
	\node (9) at (10.8, 4) {$9$};
	\fill (0.3, 3) circle (4pt);
	\node (1) at (-0.6, 3) {$1$};
	\fill (0.5, 2) circle (4pt);
	\node (6) at (-0.4, 2) {$6$};
	\fill (5, 4) circle (4pt);
	\node (7) at (4,4) {$7$};
	\fill (5, 3) circle (4pt);
	\node (3) at (4,3) {$3$};
	\fill (5, 2) circle (4pt);
	\node (5) at (4,2) {$5$};
	\fill (9.7, 3) circle (4pt);
	\node (2) at (10.6, 3) {$2$};
	\fill (9.5, 2) circle (4pt);
	\node (4) at (10.4, 2) {$4$};
\end{tikzpicture}
\caption{The stable curve corresponds to $\lim_{z \to \omega}C(z)$ on $\ovop{M}_{0,9}$}
\end{figure}
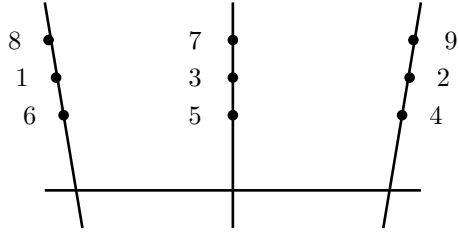

Let $C^{0} := \overline{C} = \rho(C)$ be the image of $C$ on
$\ovop{L}_{0,7}$. Then a general point of $C^{0}$ has the same
coordinates as $C$, but the limit curve is different. For example,
$\lim_{z \to \omega}C^{0}(z)$ is obtained by contracting the tail
containing $x_7$ in $\lim_{z \to \omega}C(z)$ on $\ovop{M}_{0,9}$. Let $T_{1} = \langle t \rangle$ be the one parameter subgroup corresponding to the first marked point. Then $T_{1}$-action is given by 
\small
\[
t \cdot C^{0}(z) = \left(\PP^{1}, t\cdot \alpha \frac{z-\omega}{z-\omega^{2}}, \alpha \frac{\omega z - \omega}{\omega z - \omega^{2}}, \alpha \frac{\omega^{2}z - \omega}{\omega^{2}z-\omega^{2}}, \alpha\frac{1-\omega z}{1-\omega^{2} z}, \alpha \frac{\omega - \omega z}{\omega - \omega^{2}z}, \alpha \frac{\omega^{2}-\omega z}{\omega^{2}-\omega^{2} z}, 1, 0, \infty\right)
\]
\normalsize
and $\lim_{t \to 0}t \cdot C^{0}(z)$ for general $z$ is a
rational curve with two irreducible components $X_{0}$ and
$X_{\infty}$ such that $X_{0}$ contains $x_{1}$ and $x_{8}$, and
$X_{\infty}$ contains the rest of them with the same
coordinates. Therefore on $\lim_{t \to 0}t \cdot [C^{0}]$ on $\mathrm{Chow}_{1}(\ovop{L}_{7})$,
there is an irreducible component $C_{m}$ (the so-called main
component) containing all limits of the form $\lim_{t \to 0}t \cdot
C^{0}(z)$. We see that $C_m$ is contained in $D_{\{1,8\}}$. 

\begin{figure}[!ht]
\begin{tikzpicture}[scale=0.5]
	\draw[line width=1pt] (0, 0) -- (10, 0);
	\draw[line width=1pt] (1, -1) -- (0, 5);
	\fill (0.1, 4) circle (4pt);
	\node (8) at (-0.8, 4) {$8$};
	\fill (0.5, 2) circle (4pt);
	\node (1) at (-0.4, 2) {$1$};
	\fill (3, 0) circle (4pt);
	\node (2) at (3,1) {$2$};
	\fill (4, 0) circle (4pt);
	\node (3) at (4,1) {$3$};
	\fill (5, 0) circle (4pt);
	\node (4) at (5,1) {$4$};
	\fill (6, 0) circle (4pt);
	\node (5) at (6,1) {$5$};
	\fill (7, 0) circle (4pt);
	\node (6) at (7,1) {$6$};
	\fill (8, 0) circle (4pt);
	\node (7) at (8,1) {$7$};
	\fill (9, 0) circle (4pt);
	\node (9) at (9,1) {$9$};
\end{tikzpicture}
\caption{The $(\frac{1}{7}, \frac{1}{7}, \cdots, \frac{1}{7}, 1,
  1)$-stable curve corresponds to $\lim_{z \to \omega}C^{0}(z)$ for
  general $z$ on $\ovop{L}_{7}$}
\end{figure}
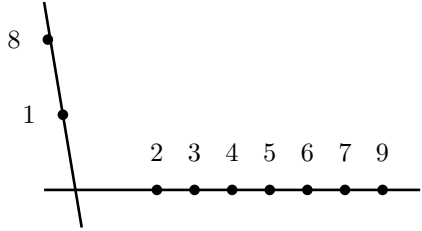

On the other hand, there are three points on $C^{0}$ that are already
contained in the toric boundary $\cup D_{I}$. These three points are
the cases where $z \to 1$, $z \to \omega$, and $z \to \omega^{2}$. If $p \in C^{0} \cap D_{I}$ and $D_{I} \cap D_{\{1,8\}} = \emptyset$, then $p^{1}:= \lim_{t \to 0}t \cdot p \notin D_{\{1,8\}}$. Therefore there must be an extra component connecting $p^{1}$ and $C_{m}$. For example, if $p := \lim_{z \to 1}C^{0}(z)$, then $p$ corresponds to a chain of rational curves $X_{0} \cup X_{1} \cup X_{\infty}$ such that $x_{2}, x_{5}, x_{8} \in X_{0}$, $x_{1}, x_{4}, x_{7} \in X_{1}$, and $x_{3}, x_{6}, x_{9} \in X_{\infty}$, or equivalently, an ordered partition $25|147|36$. Then $p^{1} = \lim_{t \to 0}t \cdot p$ corresponds to a chain corresponds to $25|1|47|36$. Similarly, $q := \lim_{z \to \omega}C^{0}(z)$ is a curve of partition type $16|357|24$ and $r := \lim_{z \to \omega^{2}}C^{0}(z)$ is a curve of partition type $34|267|15$. Hence $q^{1}$ corresponds to a curve of type $1|6|357|24$ and $r^{1}$ corresponds to a curve of type $34|267|1|5$. Note that $q^{1}$ is already on $D_{\{1,8\}}$. 

Let $z(t)$ be a holomorphic function such that $z(t) - 1$ has a simple
zero at $t = 0$. Then on $\lim_{t \to 0}t\cdot C^{0}(z(t))$, the
marked points $x_{1}$,
$x_{2}$, and $x_{5}$ approach $x_{8}$ at a constant rate, and $x_{3}$
and $x_{6}$ approach $x_{9}$ at a constant rate, too. Therefore, the
limit corresponds to a rational chain $X_{0} \cup X_{1} \cup
X_{\infty}$ such that $x_{1}, x_{2}, x_{5}, x_{8} \in X_{0}$, $x_{4},
x_{7} \in X_{1}$, and $x_{3}, x_{6}, x_{9} \in X_{\infty}$. Thus there
is a new component $C_{1}$ of $\lim_{t \to 0}t \cdot [C^{0}]$ such
that a general point of $C_{1}$ corresponds to a partition
$125|47|36$. Moreover, because the limit curve $\lim_{t \to 0}t \cdot
C^{0}$ is $T_{1}$-invariant, $C_{1}$ is $T_{1}$-invariant. So two
$T_{1}$-limits of the general point, which correspond to curves of type $1|25|47|36$ and $25|1|47|36$, are on $C^{1}$. 

If $z(t)$ is another holomorphic function such that $z(t) -
\omega^{2}$ has a simple zero at $t = 0$, then on $\lim_{t \to 0}t
\cdot C^{0}(z(t))$, the marked points $x_{3}$ and $x_{4}$ approach
$x_{8}$ at a constant rate and $x_{5}$ approaches $x_{9}$ at a
constant rate, too. So $\lim_{t \to 0}t \cdot C^{0}(z(t))$ is a curve
of partition type $34|1267|5$. Therefore, there is a new component
$C_{\omega^{2}}$ whose general point parameterizes a curve of type
$34|1267|5$. The two special points on $C_{\omega^{2}}$ parametrize curves of type $34|1|267|5$ and $34|267|1|5$. So $C_{\omega^{2}}$ does not intersect the main component $C_{m}$. 

Note that $\lim_{t \to 0}t\cdot C^{0}(z) = \lim_{t \to 0}t^{2}\cdot
C^{0}(z)$ for a general fixed $z$. But for $z(t)$ as in the previous
paragraph (that is, $z(t) -\omega^{2}$ has a simple zero at $t = 0$),
as $t \to 0$, on $t^{2}\cdot C^{0}(z(t))$, we have $x_{1}, x_{3}, x_{4} \to x_{8}$ and $x_{5} \to x_{9}$. So $\lim_{t \to 0}t \cdot C^{0}(z(t))$ is a curve of type $134|267|5$. Thus there is another component $C_{2\cdot \omega^{2}}$ of $C^{1}$ whose general point parametrizes a curve of type $134|267|5$, and its two special points parametrize curves of type $1|34|267|5$ and $34|1|267|5$, respectively. The point corresponding to the curve of type $1|34|267|5$ is on $C_{m} \subset D_{\{1,8\}}$. 

For notational convenience, we will use following notation. Let $E$ be an irreducible curve on $\Ln$, which is contained in a toric boundary $\cap D_{I}$ whose open dense subset is parameterizing curves of type $I_{1}|I_{2}|\cdots |I_{k}$. Then we say $E$ is \emph{of type $I_{1}|I_{2}|\cdots |I_{k}$}. 

In summary, the limit cycle $[C^{1}] := \lim_{t \to 0}t\cdot [C^{0}]$ for $t \in T_{1}$ has four irreducible components $C_{m}$, $C_{1}$, $C_{\omega}^{2}$, and $C_{2\cdot \omega^{2}}$ whose types are $1|234567$, $125|47|36$, $34|1267|5$, and $134|267|5$ respectively.

\subsection{Multiplicities}\label{ssec:multiplicities}

The extra irreducible components appearing on $\lim_{t \to 0}t \cdot
[C^{0}]$ may have multiplicities greater than 1. We are able to
evaluate the multiplicity of each irreducible component by computing
the number of preimages of a general point $p \in \lim_{t \to 0}t
\cdot [C^{0}]$, on $\epsilon \cdot [C^{0}]$ for small $\epsilon$. This
can be done by finding an explicit analytic germ $z(t)$ which gives the same $\lim_{t \to 0}t^{k}C^{0}(z(t))$. 

\begin{example}
For $C_{2\cdot \omega^{2}}$ in Section \ref{ssec:limitcomponents}, if we take $z(t) = \omega^{2} + \beta t + \cdots$, then on the limit cycle $X_{0} \cup X_{1} \cup X_{\infty} = \lim_{t \to 0}t^{2} C^{0}(z(t))$, the coordinates of $x_{1}$, $x_{2}$, and $x_{4}$ on $X_{0}$ are:
\[
x_{1} = \frac{\alpha(\omega^{2}-\omega)}{\beta}, 
x_{2} = \alpha \omega \beta, 
x_{4} = -\alpha \omega \beta.
\]
Since a nonzero scalar multiple gives the same cross ratio of $x_{1}$, $x_{2}$, $x_{4}$ and $0$, this configuration is equivalent to 
\[
x_{1} = \alpha(\omega^{2}-\omega), 
x_{2} = \alpha \omega \beta^{2}, 
x_{4} = -\alpha \omega \beta^{2}.
\]
Obviously $\pm \beta$ give the same limit. Thus the multiplicity is two.
\end{example}
By using a similar idea, we find that the components $C_{m}$,
$C_{1}$, $C_{\omega}^{2}$ have multiplicity 1, and $C_{2\cdot
  \omega^{2}}$ has multiplicity 2.

\subsection{Remaining steps of the computation}

In this section, for reader's convenience, we give the computation of $\overline{C} = \rho_{8,9}(C)$ for $C$ in Example \ref{ex:k=3C_00}. 

Set $C^{0} := \overline{C}$. The first limit $[C^{1}]:=\lim_{t \to 0}t \cdot [C^{0}]$ for $t \in T_{1}$ has four irreducible components, 
\[
C_{m}, C_{1}, C_{\omega^{2}}, \mbox{ and } C_{2\cdot \omega^{2}}
\]
which are of type $1|234567$, $125|47|36$, $34|1267|5$, and $134|267|5$ respectively. The component $134|267|5$ has multiplicity two. All other multiplicities are one. 

Topologically, $C^{1}$ is a tree of rational curves on
$\ovop{L}_{7}$. The main component $C_{m}$ is the spine and there are
three tails, $C_{1}$, $C_{\omega^2} \cup C_{2\cdot \omega^{2}}$, and a
`tail point' whose partition is $1|6|357|24$. (This is the point
$q^{1}$ in Section \ref{ssec:limitcomponents}.)  %As taking the limit
%with respect to the action of a one-parameter subgroup, a general
%point of the main component maps to a unique boundary stratum. But the
%limit of the three special points are not contained in the closure of
%the image of the main component in general; there will be new rational
%curves connecting them. All other components have exactly two special
%points which lie on deeper toric strata. Thus, the further
%degeneration of such rational component is just a chain of rational
%curves. Therefore, except the main component, all degenerations of
%other tails are chains of rational curves. 
  For the main component $C_m$ (which is already in $D_{\{1,8\}}$), 
there are three special points lying on the boundary of $D_{\{1,8\}}$.
If we take the limit $\lim_{t \to 0}t \cdot [C_{m}]$ for $t \in T_2$,
then except possibly for these three special points, 
all other points go to a unique boundary stratum,  $D_{\{1,8\}} \cap D_{\{1,2,8\}}$.
So if the limits of three special points are not contained in $D_{\{1,8\}} \cap D_{\{1,2,8\}}$, 
then there must be new rational curves connect the limit of general points on $C_m$ 
and the limits of special points.

For a smooth point $[(X, x_{1}, \cdots, x_{9})] \in C_{1}$, let
$Y_{2}$ be the irreducible component of $X$ containing $x_{2}$. If the
$0$-point is not on $Y_2$, let $Y_{0}$ be the \emph{connected}
component of $\overline{X - Y_{2}}$ containing $0$-point. Similarly,
$Y_{\infty}$ is the connected component of $\overline{X - Y_{2}}$
containing $\infty$-point, if $\infty$ point is not on $Y_2$. Then we
are able to evaluate limit $[C^{2}]:= \lim_{t \to 0}t \cdot [C^{1}]$
for $t \in T_{2}$ in the same way after replacing $Y_{0}$ by a
$0$-point, and $Y_{\infty}$ by an $\infty$-point because $T_2$ only
acts nontrivially on $Y_{2}$. $[C^{2}]$ has a main component (by abuse of notation, call it $C_{m}$) and three rational chains whose components are of type 
\[
	2|15|47|36, \; 12|5|47|36,
\]
\[
	1|6|2357|4, \; 1|26|357|4\; (2), 
\]
\[
	34|2|167|5, \; 34|12|67|5, \; 134|2|67|5\; (2), \; 1|234|67|5
\]
respectively. The integers in parentheses refer the multiplicity of each irreducible component. The main component is on the boundary stratum $D_{\{1,8\}} \cap D_{\{1,2,8\}}$, so it is of type $1|2|34567$. 

The limit $[C^{3}] := \lim_{t \to 0}t \cdot [C^{2}]$ for $t \in T_{3}$ has a main component of type $1|2|3|4567$ and three chains
\[
	2|15|47|3|6, \; 12|5|47|3|6, \; 1|2|5|347|6, \; 1|2|35|47|6\; (2),
\]
\[
	1|6|3|257|4, \; 1|6|23|57|4,\; 1|26|3|57|4\;(2), 1|2|36|57|4,
\]
\[
	3|4|2|167|5,\; 3|4|12|67|5,\; 3|14|2|67|5\;(2), 13|4|2|67|5\;(2), 
	1|3|24|67|5,\; 1|23|4|67|5.
\]

The next limit $[C^{4}]$ has a main component and three chains
\[
	2|15|4|7|3|6,\; 12|5|4|7|3|6,\; 1|2|5|4|37|69,\; 1|2|5|34|7|69, \; 1|2|35|4|7|6\; (2),
	1|2|3|45|7|6,
\]
\[
	1|6|3|257|4, \; 1|6|23|57|4, \; 1|26|3|57|4\;(2), \; 1|2|36|57|4, \; 1|2|3|6|457, \;
	1|2|3|46|57\;(2),
\]
and
\[
	3|4|2|167|5, \; 3|4|12|67|5, \; 3|14|2|67|5\;(2), 13|4|2|67|5\;(2), 1|3|24|67|5,
\]
\[
	 \;
	1|23|4|67|5, \; 1|2|3|4|567.
\]

Similarly, $[C^{5}]$ has a main component and three chains
\[
	2|15|4|7|3|6, \; 12|5|4|7|3|6,\; 1|2|5|4|37|69,\; 1|2|5|34|7|69, \; 1|2|35|4|7|6\; (2),
\]
\[
	1|2|3|45|7|6, \; 1|2|3|4|5|67
\]
and
\[
	1|6|3|5|27|4, \; 1|6|3|25|7|4, \; 1|6|23|5|7|4, \; 1|26|3|5|7|4\;(2), \; 1|2|36|5|7|4,
\]
\[
	1|2|3|6|5|47, \; 1|2|3|6|45|7, \; 1|2|3|46|5|7\;(2), 1|2|3|4|56|7,
\]
and 
\[
	3|4|2|167|5, \; 3|4|12|67|5, \; 3|14|2|67|5\;(2), 13|4|2|67|5\;(2), 1|3|24|67|5, \;
\]
\[
	1|23|4|67|5, \; 1|2|3|4|567.
\]

Finally, the main component of $[C^{6}]$ becomes a point, which is a torus invariant point of type $1|2|3|4|5|6|7$. Three tails are
\[
	2|15|4|7|3|6, \; 12|5|4|7|3|6,\; 1|2|5|4|37|6|9,\; 1|2|5|34|7|6|9, \; 1|2|35|4|7|6\; (2),
\]
\[
	1|2|3|45|7|6, \; 1|2|3|4|5|67
\]
and
\[
	1|6|3|5|27|4, \; 1|6|3|25|7|4, \; 1|6|23|5|7|4, \; 1|26|3|5|7|4\;(2), \; 1|2|36|5|7|4,
\]
\[
	1|2|3|6|5|47, \; 1|2|3|6|45|7, \; 1|2|3|46|5|7\;(2), 1|2|3|4|56|7,
\]
and 
\[
	3|4|2|6|17|5, \; 3|4|2|16|7|5, \; 3|4|12|6|7|5, \; 3|14|2|6|7|5\;(2), 13|4|2|6|7|5\;(2), 
\]
\[
	1|3|24|6|7|5, \; 1|23|4|6|7|5, \; 1|2|3|4|6|57, \; 1|2|3|4|56|7.
\]

Now we are able to describe each one dimensional boundary stratum as the image of an F-curve on $\ovop{M}_{0,9}$. By abuse of notation, let $F_{I_{1}, I_{2}, I_{3}, I_{4}}$ be the image of an F-curve on $\ovop{L}_{7} \cong \ovop{M}_{0, \left((\frac{1}{7})^{n}, 1, 1\right)}$ of $F_{I_{1}, I_{2}, I_{3}, I_{4}}$. The image $F_{I_{1}, I_{2}, I_{3}, I_{4}}$ is a torus invariant curve if and only if $8 \in I_{1}$, $9 \in I_{4}$, $|I_{2}| = |I_{3}| = 1$. It is contracted if and only if $8, 9 \in I_{1}$ and otherwise the image is not a torus invariant curve. For example, the torus invariant stratum $2|15|4|7|3|6$ is $F_{\{82,1,5,34679\}}$. Thus we obtain 
\begin{eqnarray*}
	\overline{C} = C^{0} &\equiv& 
	F_{\{82, 1,5, 34679\}} + F_{\{8, 1, 2, 345679\}}
	+ F_{\{81245, 3, 7, 69\}} + F_{\{8125, 3, 4,  679\}}\\
	&& + 2F_{\{812,3,5,4679\}} + F_{\{8123, 4, 5, 679\}} 
	+ F_{\{812345, 6, 7, 9\}} + F_{\{81356,2,7,49\}}\\
	&& + F_{\{8136,2,5,479\}} + F_{\{816,2,3,4579\}}
	+ 2F_{\{81,2,6,34579\}} + F_{\{812,3,6,4579\}}\\
	&& + F_{\{812356,4,7,9\}} + F_{\{81236, 4, 5, 79\}} 
	+ 2F_{\{8123,4,6,579\}} + F_{\{81234,5,6,79\}}\\
	&& + F_{\{82346,1,7,59\}} + F_{\{8234,1,6,579\}}
	+ F_{\{834,1,2,5679\}} + 2F_{\{83,1,4,25679\}}\\
	&& + 2F_{\{8,1,3,245679\}} + F_{\{813,2,4,5679\}}
	+ F_{\{81,2,3,45679\}} + F_{\{812346,5,7,9\}}\\
	&& + F_{\{81234,5,6,79\}}.
\end{eqnarray*}
on $\ovop{L}_{7}$. 

\subsection{Summary of the computation}\label{ssec:summary}

Here we summarize the strategy used in this section. Let $C$ be an
irreducible curve on $\Ln$, which intersects the big cell of
$\Ln$. Let $T_{i}$ be the one parameter subgroup moving $x_{i}$ only.
Let $L(C,n,i)$ denote the procedure to evaluate the limit cycle $\lim_{t \to 0}t \cdot C$ for $t \in T_{i}$ on $\Ln$. 

\begin{algorithm}[$L(C, n, i)$]
Let $C$ be an irreducible curve on $\Ln$. 
\begin{enumerate}
\item Write coordinates $(\PP^{1}, x_{1}(z), x_{2}(z), \cdots,
  x_{n}(z))$ of a general point on $C$, such that the $0$-point is $0$
  and the $\infty$-point is $\infty$.
\item Find all special points $p_{1}, p_{2}, \cdots, p_{k}$ on $C \cap (\Ln - (\CC^{*})^{n-1})$. Suppose that $p_{j}$ occurs when $z = z_{j}$. 
\item Take $\lim_{t \to 0}t \cdot C(z)$ for $t \in T_{i}$ and general $z \in C$. The closure is the main component $C_{m}$.
\item For each $p_{j}$, find all limits of the form $\lim_{t \to 0}t \cdot C(z(t))$ where $z(t)$ is a holomorphic function such that $z(t) - z_{j}$ has a pole of order $r$. Take the closure of all such limits and obtain irreducible components $C_{r\cdot z_{j}}$ connecting $\lim_{t \to 0}t \cdot p_{j}$ and $C_{m}$. 
\item Evaluate the multiplicity of each irreducible component $C_{r
    \cdot z_{j}}$, by counting the number of preimages of a general point $p \in C_{r \cdot z_{j}}$ on $\epsilon \cdot C_{r \cdot z_{j}}$. 
\end{enumerate}
\end{algorithm}

Then we can evaluate the toric degeneration by applying the algorithm $L(C, n, i)$ several times. 

\begin{algorithm}[Evaluation of the limit cycle]
Set $C^{0} = C$ and $i = 1$.
\begin{enumerate}
\item Write $C^{i-1} = \sum m_{j}C_{j}$ as a linear combination of irreducible components. 
\item Each irreducible component $C_{j}$ lies on a boundary stratum, which is isomorphic to $\prod \overline{\mathrm{L}}_{k}$. Furthermore, there is a unique $\overline{\mathrm{L}}_{k}$ where $x_{i+1}$ is not forgotten. 
\item Apply $L(C_{j}, k, i)$ to each irreducible component and set $C^{i}$ as the formal sum of all limits of irreducible components. 
\item Set $i = i + 1$ and repeat (1) and (2) to evaluate $C^{i}$ for $2 \le i \le n-1$.
\item $C^{n-1}$ is the desired toric degeneration. 
\end{enumerate}
\end{algorithm}

\section{Finding an effective linear combination}\label{sec:efflincomb}

Let $C$ be an effective rational curve on $\ovop{M}_{0,n+2}$ and
$\overline{C}$ be its image in $\Ln$. Using the techniques of the
previous section, we are able to compute the numerical class of
$\overline{C}$ on $\Ln$ as an effective linear combination for toric
boundary curves. We can regard it as a first approximation of an
effective linear combination of F-curves of $C$. In this section, we
will discuss some computational ideas to write $C$ as an effective $\ZZ$-linear combination of F-curves. 

On $\Ln$, suppose that $\overline{C} \equiv \sum b_{I}F_{I}$ where
$b_{I} > 0$ and $F_{I}$ is a torus invariant curve which is the image
of an F-curve, which, abusing notation, we also denote
$F_{I}$. Consider $\sum b_{I}F_{I}$ on $\ovop{M}_{0,n+2}$. In general,
it is not numerically equivalent to $C$ because $C$ passes through
several exceptional loci of $\rho : \ovop{M}_{0,n+2} \to \Ln$. Since
$\rho$ is a composition of smooth blow-downs, we are able to compute
the curve classes on the Mori cone of exceptional fiber passing
through $C$, which we need to subtract from $\sum b_{I}F_{I}$. This yields a numerical class of $C$ of the form 
\[
	C \equiv \sum b_{I}F_{I} - \sum c_{J}F_{J}
\]
where $b_{I}, c_{J} > 0$. 

\begin{example}\label{ex:linearcombinationn=9}
Consider the curve $C$ in Example \ref{ex:dihedralcurve}. Note that for
$\rho_{8,9} : \ovop{M}_{0,9} \to \ovop{L}_{7}$, all the exceptional loci intersecting $C$ are F-curves. By considering the proper transform, on $\ovop{M}_{0,9}$, we have 
\begin{eqnarray*}
	C &\equiv& F_{\{82, 1,5, 34679\}} + F_{\{8, 1, 2, 345679\}}
	+ F_{\{81245, 3, 7, 69\}} + F_{\{8125, 3, 4,  679\}}\\ 
	&& + 2F_{\{812,3,5,4679\}} + F_{\{8123, 4, 5, 679\}}
	+ F_{\{812345, 6, 7, 9\}} + F_{\{81356,2,7,49\}}\\
	&& + F_{\{8136,2,5,479\}} + F_{\{816,2,3,4579\}}
	+ 2F_{\{81,2,6,34579\}} + F_{\{812,3,6,4579\}}\\
	&& + F_{\{812356,4,7,9\}} + F_{\{81236, 4, 5, 79\}}
	+ 2F_{\{8123,4,6,579\}} + F_{\{81234,5,6,79\}}\\
	&& + F_{\{82346,1,7,59\}} + F_{\{8234,1,6,579\}}
	+ F_{\{834,1,2,5679\}} + 2F_{\{83,1,4,25679\}}\\
	&& + 2F_{\{8,1,3,245679\}} + F_{\{813,2,4,5679\}}
	+ F_{\{81,2,3,45679\}} + F_{\{812346,5,7,9\}}\\
	&& + F_{\{81234,5,6,79\}}\\
	&& - 2F_{\{1,2,3,456789\}} - 2F_{\{4,5,6,123789\}}
	- F_{\{1,4,7,235689\}} - F_{\{2,6,7,134589\}}\\
	&& - F_{\{3,5,7,124689\}}.
\end{eqnarray*}
\end{example}

To make the given linear combination of F-curves effective, we need to add some numerically trivial linear combination of F-curves. By \cite[Theorem 7.3]{KM94}, the vector space of numerically trivial curve classes on $\Mzn$ is generated by Keel relations. 

\begin{definition}\cite[Lemma 7.2.1]{KM94}
Let $I_{1} \sqcup I_{2} \sqcup I_{3} \sqcup I_{4} \sqcup I_{5}$ be a partition of $[n]$. Then the following linear combination of F-curves is numerically trivial:
\[
	F_{I_{1}, I_{2}, I_{3}, I_{4}\sqcup I_{5}} + F_{I_{1}\sqcup I_{2}, I_{3}, I_{4}, I_{5}}
	- F_{I_{1}, I_{4}, I_{3}, I_{2}\sqcup I_{5}} - F_{I_{1}\sqcup I_{4}, I_{3}, I_{2}, I_{5}}
\]
We call relations of this form \emph{Keel relations} among F-curves.
\end{definition}
Note that in a Keel relation, all F-curves share a common set
$I_{3}$. Moreover, two F-curves with the same sign share exactly one common set,
and two F-curves with different sign share exactly two common sets. (For
instance, $F_{I_{1}, I_{2}, I_{3}, I_{4}\sqcup I_{5}}$ and $F_{I_{1},
  I_{4}, I_{3}, I_{2}\sqcup I_{5}}$ have common sets $I_{1}$ and
$I_{3}$.). Conversely, this is a necessary and sufficient condition
for the existence of a Keel relation containing certain F-curves.

Let $F_{I}$ and $F_{J}$ be two F-curves on $\Mzn$, where $I := \{I_{1}, I_{2}, I_{3}, I_{4}\}$ and $J:= \{J_{1}, J_{2}, J_{3}, J_{4}\}$. The \emph{common refinement} $R_{I, J}$ of two partitions $I$ and $J$ is the set of nonempty subsets of $[n]$ of the form $I_{i} \cap J_{j}$ for $1 \le i, j \le 4$. And the \emph{intersection} $S_{I,J}$ of $I$ and $J$ is the set of all nonempty subsets $K \subset [n]$ such that $K = I_{i}$ for some $i$ and $K = J_{j}$ for some $j$, too.

\begin{definition}
We say two F-curves $F_{I}, F_{J}$ on $\Mzn$ are \emph{adjacent} if $|R_{I, J}| = 5$. 
\end{definition}
(The motivation for this terminology will become clear below.)

\begin{lemma}
Let $F_{I}$ and $F_{J}$ be two F-curves on $\Mzn$. 
\begin{enumerate}
\item There is a Keel relation containing $F_{I}$ and $F_{J}$ if and only if $F_{I}$ and $F_{J}$ are adjacent.
\item In this case, the number of Keel relations containing both $F_{I}$ and $F_{J}$ is two.
\item If $|S_{I, J}| = 2$, then the signs of $F_{I}$ and $F_{J}$ in a Keel relation containing them are different. 
\item If $|S_{I,J}| = 1$, then the signs of $F_{I}$ and $F_{J}$ in a
  Keel relation containing them are same. 
\item Suppose $F_{I}, F_{J}, F_{K}$ are pairwise adjacent and $|S_{I, J}| =
  2$, $|S_{I, K}| = 2$, and $|S_{J,K}| = 1$.  Then there is a unique Keel relation containing all of them.
\end{enumerate}
\end{lemma}

\begin{proof}
Straightforward computation.
\end{proof}

Next, we describe two families of graphs.  
\begin{definition}
Let $E$ be a $\ZZ$-linear combination of F-curves.   We define a
rooted infinite graph $G(E)$ as follows.  
\begin{enumerate}
\item The set of vertices of $G(E)$ is the infinite set of expressions equivalent
  to $E$.  
\item The root is the vertex $E$.  
\item Two vertices $E'$ and $E''$ are connected by an edge if $E'' = E' +
  R$ for some Keel relation $R$.
\end{enumerate}
For each nonnegative integer $l$, let $G(E,l)$ be the subgraph of
$G(E)$ consisting of vertices that are connected to the root by a path
of length less than or equal to $l$.
\end{definition}

We will restrict our attention to a smaller graph $\widetilde{G}(E)$.  It has
the same vertices as $G(E)$, but fewer edges:
\begin{definition}
Let $E$ be a $\ZZ$-linear combination of F-curves.   Write $E = \sum_{I \in \mathcal{I}} b_{I}F_{I} - \sum_{J \in \mathcal{J}}  c_{J}F_{J}$ with $b_J, c_J > 0$ for all $I \in \mathcal{I}$ and $J \in \mathcal{J}$.   We define a rooted graph $\widetilde{G}(E)$ as follows.  
\begin{enumerate}
\item The set of vertices of $\widetilde{G}(E)$ is the (infinite) set of expressions equivalent  to $E$.  
\item The root is the vertex $E$.  
\item Two vertices $E'$ and $E''$ are connected by an edge if $E'' = E' + R$, where $R$ is a Keel relation containing at least one positive curve $F_{I}$ from $E'$ and at least one negative curve $F_{J}$ from $E'$.
\end{enumerate}
For each nonnegative integer $l$, let $\widetilde{G}(E,l)$ be the subgraph of
$\widetilde{G}(E)$ consisting of vertices that are connected to the root by a path
of length less than or equal to $l$.
\end{definition}

We are now ready to describe our strategy for finding effective
expressions for curve classes:
\begin{strategy} \label{strategy:find effective sum}
Let $E = \sum_{I \in \mathcal{I}} b_{I}F_{I} - \sum_{J \in \mathcal{J}}  c_{J}F_{J}$ be a $\ZZ$-linear combination
of F-curves with $b_I, c_J > 0$ for all $I \in \mathcal{I}$ and $J \in
\mathcal{J}$.  

\begin{enumerate}
\item Let $m(E) = \sum_{J \in \mathcal{J}} c_J$.  We use the integer $m(E)$ as a measure  of how far the $E$ expression is from being effective.
\item Beginning with $l=1$, compute $\widetilde{G}(E,l)$.  If $\widetilde{G}(E,l)$ contains a vertex
  $E'$ corresponding to an expression $E' = \sum_{I \in \mathcal{I}'}
  b_{I}'F_{I} - \sum_{J \in \mathcal{J}'}  c_{J}'F_{J}$ with 
\[ m(E') = \sum_{J \in \mathcal{J}'} c_J' < m(E) = \sum_{J \in \mathcal{J}} c_J,
\]
start over again replacing $E$
  by $E'$.  If $m(E') = m(E)$ for all $E' \in \widetilde{G}(E,l)$,
  repeat this step with $l=l+1$.    
\item Continue until an effective expression ($m(E') = 0$) is found.
\end{enumerate}
\end{strategy}

This strategy is implemented in the \texttt{M0nbar} package for
\texttt{Macaulay2} in the command \texttt{seekEffectiveExpression}.  

Strategy \ref{strategy:find effective sum} is not an algorithm because it is not
guaranteed to produce an effective expression, even if an effective
expression is known to exist; for an example where the
strategy fails, see the calculations for the $D_4$ fixed
curve on $\ovop{M}_{0,12}$ at the link below to the second author's
webpage.  Nevertheless, although our strategy is not an algorithm, we
were still able to use it successfully to check that all the curves
$\Mzn^G$ for $G$ dihedral and $n \leq 12$ are effective linear combinations of F-curves.

By Lemmas \ref{lem:dihedralcurvesnk} and
\ref{lem:whendihedraliscyclic}, the curves $\Mzn^{D_k}$ which are not
of the form $C^{\sigma}$ are:
\begin{center} 
\begin{tabular}{ll}
$n=9$ & $k=3$ \\
$n=11$ & $k=3$\\
$n=12$ & $k=3$\\
$n=12$ & $k=4$.
\end{tabular}
\end{center}
Moreover, when $n=12$, there is also a curve of the form
$\Mzn^{A_4}$.  

We present our calculations for two examples ($n=9$ and $k=3$, and
$n=12$ and $k=3$) here in the paper.  The remaining calculations can
be found on our website for this project:
\begin{center}
http://faculty.fordham.edu/dswinarski/invariant-curves/
\end{center}

\begin{example}
Let $C$ be the curve in Example \ref{ex:k=3C_00}. We  found a
noneffective $\ZZ$-linear combination in Example
\ref{ex:linearcombinationn=9}. By using Strategy \ref{strategy:find
  effective sum}, we can find Keel relations which make the linear combination into an
effective one. In the following expression, each term in parentheses
is a Keel relation.
\begin{eqnarray*}
	C &\equiv& F_{\{82, 1,5, 34679\}} + F_{\{8, 1, 2, 345679\}}
	+ F_{\{81245, 3, 7, 69\}} + F_{\{8125, 3, 4,  679\}}\\
	&& + 2F_{\{812,3,5,4679\}} + F_{\{8123, 4, 5, 679\}}
	+ F_{\{812345, 6, 7, 9\}} + F_{\{81356,2,7,49\}}\\
	&& + F_{\{8136,2,5,479\}} + F_{\{816,2,3,4579\}}
	+ 2F_{\{81,2,6,34579\}} + F_{\{812,3,6,4579\}}\\
	&& + F_{\{812356,4,7,9\}} + F_{\{81236, 4, 5, 79\}}
	+ 2F_{\{8123,4,6,579\}} + F_{\{81234,5,6,79\}}\\
	&& + F_{\{82346,1,7,59\}} + F_{\{8234,1,6,579\}}
	+ F_{\{834,1,2,5679\}} + 2F_{\{83,1,4,25679\}}\\
	&& + 2F_{\{8,1,3,245679\}} + F_{\{813,2,4,5679\}}
	+ F_{\{81,2,3,45679\}} + F_{\{812346,5,7,9\}}\\
	&& + F_{\{81234,5,6,79\}}\\
	&&- 2F_{\{1,2,3,456789\}} - 2F_{\{4,5,6,123789\}}
	- F_{\{1,4,7,235689\}} - F_{\{2,6,7,134589\}}\\
	&& - F_{\{3,5,7,124689\}}\\
	&& + \left(F_{\{1,2,3,456789\}} + F_{\{13,2,8,45679\}}
	- F_{\{1,2,8,345679\}} - F_{\{81, 2, 3, 45679\}}\right)\\
	&& + \left(F_{\{4,5,6,123789\}} + F_{\{8123, 45, 6, 79\}}
	- F_{\{8123,4,6,579\}} - F_{\{81234, 5, 6, 79\}}\right)\\
	&& + \left(F_{\{4,5,6,123789\}} + F_{\{4,56,79,8123\}}
	- F_{\{81236,4,5,79\}} - F_{\{8123,4,6,579\}}\right)\\
	&& + \left(F_{\{812,4, 5, 3679\}} + F_{\{8124, 3, 5, 679\}}
	- F_{\{812,3,5,4679\}} - F_{\{8123, 4, 5, 679\}}\right)\\
	&& + \left(F_{\{8124, 5, 6, 379\}} + F_{\{81246, 3, 5, 79\}} 
	- F_{\{8124, 3, 5, 679\}} - F_{\{81234, 5, 6, 79\}}\right)\\
	&& + \left(F_{\{812469, 3,5,7\}} + F_{\{81246, 5, 9, 37\}} 
	- F_{\{81246, 3, 5, 79\}} - F_{\{812346, 5, 7, 9\}}\right)\\
	&& + \left(F_{\{81,3,6,24579\}} + F_{\{813,2,6,4579\}}
	- F_{\{81,2,6,34579\}} - F_{\{812,3,6,4579\}}\right)\\
	&& + \left(F_{\{813,6,45,279\}} + F_{\{81345, 2, 6, 79\}}
	- F_{\{813,2,6,4579\}} - F_{\{8123,6,45,79\}}\right)\\
	&& + \left(F_{\{81345, 6, 9, 27\}} + F_{\{813459, 2, 6, 7\}}
	- F_{\{81345, 2, 6, 79\}} - F_{\{812345, 6, 7, 9\}}\right)\\
	&& +\left(F_{\{823,1,4,5679\}} + F_{\{83,1,2,45679\}}
	- F_{\{834,1,2,5679\}} - F_{\{83,1,4,25679\}}\right)\\
	&& + \left(F_{\{1,8,23, 45679\}} + F_{\{1,2,3,456789\}}
	- F_{\{83,1,2,45679\}} - F_{\{1,3,8, 245679\}}\right)\\
	&& + \left(F_{\{83,2,4,15679\}} + F_{\{823,1,4,5679\}}
	- F_{\{83,1,4,25679\}} - F_{\{813,2,4,5679\}}\right)\\
	&& + \left(F_{\{823,4,56,179\}} + F_{\{82356,1,4,79\}}
	- F_{\{823,1,4,5679\}} - F_{\{8123,4,56,79\}}\right)\\
	&& + \left(F_{\{82356,4,17,9\}} + F_{\{1,4,7,235689\}}
	- F_{\{82356, 1, 4, 79\}} - F_{\{812356, 4, 7, 9\}}\right)
\end{eqnarray*}
\begin{eqnarray*}
	&\equiv& F_{\{82, 1,5, 34679\}} + F_{\{81245, 3, 7, 69\}} 
	+ F_{\{8125, 3, 4,  679\}} + F_{\{812,3,5,4679\}}\\
	&& + F_{\{81356,2,7,49\}} + F_{\{8136,2,5,479\}}
	+ F_{\{816,2,3,4579\}} + F_{\{81,2,6,34579\}}\\
	&& + F_{\{82346,1,7,59\}} + F_{\{8234,1,6,579\}}
	+ F_{\{8,1,3,245679\}} + F_{\{13,2,8,45679\}}\\
	&& + F_{\{812,4, 5, 3679\}} + F_{\{8124, 5, 6, 379\}}
	+ F_{\{81246, 5, 9, 37\}} + F_{\{81,3,6,24579\}}\\ 
	&& + F_{\{813,6,45,279\}} + F_{\{81345, 6, 9, 27\}}
	+ F_{\{823,1,4,5679\}} + F_{\{8,1,23, 45679\}} \\
	&& + F_{\{83,2,4,15679\}} + F_{\{823,4,56,179\}} 
	+ F_{\{82356,4,17,9\}}.
\end{eqnarray*}
\end{example}

\section{An example on $\ovop{M}_{0,12}$}\label{sec:example}

In this section, we compute a numerically equivalent effective linear combination of F-curves for Example \ref{ex:dihedraln=12}. Recall that in the example, 
\[
	G = \langle (123)(456)(789)(abc), (14)(26)(35)(89)(bc)\rangle \cong D_{3},
\]
$G$ is of type $(0, 2)$. Let $C = \ovop{M}_{0,12}^{G}\subset \ovop{M}_{0,12}$.

By using the method of Section \ref{sec:LosevManin}, on $\ovop{L}_{10}$, we have
\begin{eqnarray*}
\overline{C}&\equiv& F_{\{82,1,5,3467abc9\}} + F_{\{8,1,2,34567abc9\}} + F_{\{812457ab,3,c,69\}}
+ F_{\{812457a,3,b,6c9\}}\\
&& + F_{\{812457,3,a,6bc9\}} + F_{\{81245, 3,7, 6abc9\}} + F_{\{8125,3,4,67abc9\}} + 2F_{\{812,3,5,467abc9\}}\\
&& + F_{\{8123,4,5,67abc9\}} + F_{\{8123457ab,6,c,9\}} + F_{\{8123457a,6,b,c9\}} + F_{\{8123457, 6, a, bc9\}}\\
&& + F_{\{812345, 6, 7, abc9\}}\\
%\end{eqnarray*}
%\begin{eqnarray*}
&& + F_{\{813567ab,2,c,49\}} + F_{\{813567a,2,b,4c9\}} + F_{\{813567,2,a,4bc9\}} + F_{\{81356,2,7,4abc9\}}\\
&& + F_{\{8136,2,5,47abc9\}} + F_{\{816,2,3,457abc9\}} + 2F_{\{81,2,6,3457abc9\}} + F_{\{812,3,6,457abc9\}}\\
&& + F_{\{8123567ab,4,c,9\}} + F_{\{8123567a,4,b,c9\}} + F_{\{8123567,4,a,bc9\}} + F_{\{812356,4,7,abc9\}}\\
&& + F_{\{81236,4,5,7abc9\}} + 2F_{\{8123,4,6,57abc9\}} + F_{\{81234, 5, 6, 7abc9\}}\\
%\end{eqnarray*}
%\begin{eqnarray*}
&& + F_{\{823467ab,1,c,59\}} + F_{\{823467a,1,b,5c9\}} + F_{\{823467, 1, a, 5bc9\}} + F_{\{82346, 1, 7, 5abc9\}}\\
&& + F_{\{8234,1,6,57abc9\}} + F_{\{834,1,2,567abc9\}} + 2F_{\{83,1,4,2567abc9\}} + 2F_{\{8,1,3,24567abc9\}}\\
&& + F_{\{813,2,4,567abc9\}} + F_{\{81,2,3,4567abc9\}} + F_{\{8123467ab,5,c,9\}} + F_{\{8123467a,5,b,c9\}}\\
&& + F_{\{8123467,5,a,bc9\}} + F_{\{812346,5,7,abc9\}} + F_{\{81234,5,6,7abc9\}}.
\end{eqnarray*}
In this degeneration, there are three rational tails. Rows 1--4 in the
expression above form the first tail, rows 4--8 form a second
tail, and rows 9--12 form the third tail.   

By considering the proper transform, we deduce that on $\ovop{M}_{0,12}$,
\begin{eqnarray*}
\overline{C}&\equiv& F_{\{82,1,5,3467abc9\}} + F_{\{8,1,2,34567abc9\}} + F_{\{812457ab,3,c,69\}}
+ F_{\{812457a,3,b,6c9\}}\\
&& + F_{\{812457,3,a,6bc9\}} + F_{\{81245, 3,7, 6abc9\}} + F_{\{8125,3,4,67abc9\}} + 2F_{\{812,3,5,467abc9\}}\\
&& + F_{\{8123,4,5,67abc9\}} + F_{\{8123457ab,6,c,9\}} + F_{\{8123457a,6,b,c9\}} + F_{\{8123457, 6, a, bc9\}}\\
&& + F_{\{812345, 6, 7, abc9\}}\\
%\end{eqnarray*}
%\begin{eqnarray*}
&& + F_{\{813567ab,2,c,49\}} + F_{\{813567a,2,b,4c9\}} + F_{\{813567,2,a,4bc9\}} + F_{\{81356,2,7,4abc9\}}\\
&& + F_{\{8136,2,5,47abc9\}} + F_{\{816,2,3,457abc9\}} + 2F_{\{81,2,6,3457abc9\}} + F_{\{812,3,6,457abc9\}}\\
&& + F_{\{8123567ab,4,c,9\}} + F_{\{8123567a,4,b,c9\}} + F_{\{8123567,4,a,bc9\}} + F_{\{812356,4,7,abc9\}}\\
&& + F_{\{81236,4,5,7abc9\}} + 2F_{\{8123,4,6,57abc9\}} + F_{\{81234, 5, 6, 7abc9\}}\\
%\end{eqnarray*}
%\begin{eqnarray*}
&& + F_{\{823467ab,1,c,59\}} + F_{\{823467a,1,b,5c9\}} + F_{\{823467, 1, a, 5bc9\}} + F_{\{82346, 1, 7, 5abc9\}}\\
&& + F_{\{8234,1,6,57abc9\}} + F_{\{834,1,2,567abc9\}} + 2F_{\{83,1,4,2567abc9\}} + 2F_{\{8,1,3,24567abc9\}}\\
&& + F_{\{813,2,4,567abc9\}} + F_{\{81,2,3,4567abc9\}} + F_{\{8123467ab,5,c,9\}} + F_{\{8123467a,5,b,c9\}}\\
&& + F_{\{8123467,5,a,bc9\}} + F_{\{812346,5,7,abc9\}} + F_{\{81234,5,6,7abc9\}}\\
%\end{eqnarray*}
%\begin{eqnarray*}
&& - 2F_{\{1,2,3,456789abc\}} - 2F_{\{4,5,6,123789abc\}} - F_{\{1,4,7,235689abc\}} - F_{\{3,5,7,124689abc\}}\\
&& - F_{\{2,6,7,134589abc\}} - F_{\{1,5,c,2346789ab\}} - F_{\{2,6,a,1345789bc\}} - F_{\{3,4,b,1256789ac\}} \\
&& - F_{\{1,4,a,2356789bc\}}- F_{\{2,5,b,1346789ac\}} - F_{\{3,6,c,1245789ab\}} - F_{\{1,6,b,2345789ac\}}\\
&&  - F_{\{2,4,c,1356789ab\}} - F_{\{3,5,a,1246789bc\}}.
\end{eqnarray*}
Using the \texttt{seekEffectiveExpression} command in the
\texttt{M0nbar} package for \texttt{Macaulay2}, we obtain 
\begin{eqnarray*}
%--March 19
\overline{C}&\equiv& F_{\{82,1,5,3467abc9\}} + F_{\{8,1,2,34567abc9\}} + F_{\{812457ab,3,c,69\}}
+ F_{\{812457a,3,b,6c9\}}\\
&& + F_{\{812457,3,a,6bc9\}} + F_{\{81245, 3,7, 6abc9\}} + F_{\{8125,3,4,67abc9\}} + 2F_{\{812,3,5,467abc9\}}\\
&& + F_{\{8123,4,5,67abc9\}} + F_{\{8123457ab,6,c,9\}} + F_{\{8123457a,6,b,c9\}} + F_{\{8123457, 6, a, bc9\}}\\
&& + F_{\{812345, 6, 7, abc9\}}\\
&& + F_{\{813567ab,2,c,49\}} + F_{\{813567a,2,b,4c9\}} + F_{\{813567,2,a,4bc9\}} + F_{\{81356,2,7,4abc9\}}\\
&& + F_{\{8136,2,5,47abc9\}} + F_{\{816,2,3,457abc9\}} + 2F_{\{81,2,6,3457abc9\}} + F_{\{812,3,6,457abc9\}}\\
&& + F_{\{8123567ab,4,c,9\}} + F_{\{8123567a,4,b,c9\}} + F_{\{8123567,4,a,bc9\}} + F_{\{812356,4,7,abc9\}}\\
&& + F_{\{81236,4,5,7abc9\}} + 2F_{\{8123,4,6,57abc9\}} + F_{\{81234, 5, 6, 7abc9\}}\\
&& + F_{\{823467ab,1,c,59\}} + F_{\{823467a,1,b,5c9\}} + F_{\{823467, 1, a, 5bc9\}} + F_{\{82346, 1, 7, 5abc9\}}\\
&& + F_{\{8234,1,6,57abc9\}} + F_{\{834,1,2,567abc9\}} + 2F_{\{83,1,4,2567abc9\}} + 2F_{\{8,1,3,24567abc9\}}\\
&& + F_{\{813,2,4,567abc9\}} + F_{\{81,2,3,4567abc9\}} + F_{\{8123467ab,5,c,9\}} + F_{\{8123467a,5,b,c9\}}\\
&& + F_{\{8123467,5,a,bc9\}} + F_{\{812346,5,7,abc9\}} + F_{\{81234,5,6,7abc9\}}\\
&& - 2F_{\{1,2,3,456789abc\}} - 2F_{\{4,5,6,123789abc\}} - F_{\{1,4,7,235689abc\}} - F_{\{3,5,7,124689abc\}}\\
&& - F_{\{2,6,7,134589abc\}} - F_{\{1,5,c,2346789ab\}} - F_{\{2,6,a,1345789bc\}} - F_{\{3,4,b,1256789ac\}} \\
&& - F_{\{1,4,a,2356789bc\}}- F_{\{2,5,b,1346789ac\}} - F_{\{3,6,c,1245789ab\}} - F_{\{1,6,b,2345789ac\}}\\
&&  - F_{\{2,4,c,1356789ab\}} - F_{\{3,5,a,1246789bc\}}\\
&& + \left(F_{\{1238, 4, 56, 79abc\}} - F_{\{1238, 4, 579abc, 6\}} - F_{\{12368, 4, 5, 79abc\}} + F_{\{123789abc, 4, 5, 6\}}\right) \\
&& + \left(F_{\{1238, 45, 6, 79abc\}} - F_{\{1238, 4, 579abc, 6\}} - F_{\{12348, 5, 6, 79abc\}} + F_{\{123789abc, 4, 5, 6\}}\right) \\
&& + \left(F_{\{124578ab, 36, 9, c\}} - F_{\{124578ab, 3, 69, c\}} - F_{\{1234578ab, 6, 9, c\}} + F_{\{1245789ab, 3, 6, c\}}\right) \\
&& + \left(F_{\{12, 3, 45679abc, 8\}} - F_{\{1, 245679abc, 3, 8\}} - F_{\{18, 2, 3, 45679abc\}} + F_{\{1, 2, 3, 456789abc\}}\right) \\
&& + \left(F_{\{135678ab, 24, 9, c\}} - F_{\{1235678ab, 4, 9, c\}} - F_{\{135678ab, 2, 49, c\}} + F_{\{1356789ab, 2, 4, c\}}\right) \\
&& + \left(F_{\{15, 234678ab, 9, c\}} - F_{\{1, 234678ab, 59, c\}} - F_{\{1234678ab, 5, 9, c\}} + F_{\{1, 2346789ab, 5, c\}}\right) \\
&& + \left(F_{\{1248, 3, 57, 69abc\}} - F_{\{12458, 3, 69abc, 7\}} - F_{\{1248, 3, 5, 679abc\}} + F_{\{124689abc, 3, 5, 7\}}\right) \\
&& + \left(F_{\{128, 3679abc, 4, 5\}} - F_{\{1238, 4, 5, 679abc\}} - F_{\{128, 3, 4679abc, 5\}} + F_{\{1248, 3, 5, 679abc\}}\right) \\
&& + \left(F_{\{1, 23, 45679abc, 8\}} - F_{\{1, 245679abc, 3, 8\}} - F_{\{1, 2, 38, 45679abc\}} + F_{\{1, 2, 3, 456789abc\}}\right) \\
&& + \left(F_{\{1, 238, 4, 5679abc\}} - F_{\{1, 25679abc, 38, 4\}} - F_{\{1, 2, 348, 5679abc\}} + F_{\{1, 2, 38, 45679abc\}}\right) \\
&& + \left(F_{\{1, 2368, 47, 59abc\}} - F_{\{1, 23468, 59abc, 7\}} - F_{\{1, 2368, 4, 579abc\}} + F_{\{1, 235689abc, 4, 7\}}\right) \\
&& + \left(F_{\{1, 238, 4579abc, 6\}} - F_{\{1, 2348, 579abc, 6\}} - F_{\{1, 238, 4, 5679abc\}} + F_{\{1, 2368, 4, 579abc\}}\right) \\
&& + \left(F_{\{1368, 2, 479ac, 5b\}} - F_{\{1368, 2, 479abc, 5\}} - F_{\{13568, 2, 479ac, b\}} + F_{\{1346789ac, 2, 5, b\}}\right) \\
&& + \left(F_{\{13568b, 2, 49c, 7a\}} - F_{\{135678a, 2, 49c, b\}} - F_{\{13568, 2, 49bc, 7a\}} + F_{\{13568, 2, 479ac, b\}}\right) \\
&& + \left(F_{\{1345689bc, 2, 7, a\}} - F_{\{13568, 2, 49abc, 7\}} - F_{\{135678, 2, 49bc, a\}} + F_{\{13568, 2, 49bc, 7a\}}\right) \\
&& + \left(F_{\{1, 238, 4579ac, 6b\}} - F_{\{1, 238, 4579abc, 6\}} - F_{\{1, 2368, 4579ac, b\}} + F_{\{1, 2345789ac, 6, b\}}\right) \\
&& + \left(F_{\{1, 2368b, 47, 59ac\}} - F_{\{1, 2368, 47, 59abc\}} - F_{\{1, 234678, 59ac, b\}} + F_{\{1, 2368, 4579ac, b\}}\right) \\
&& + \left(F_{\{1, 234678b, 59c, a\}} - F_{\{1, 234678, 59bc, a\}} - F_{\{1, 234678a, 59c, b\}} + F_{\{1, 234678, 59ac, b\}}\right) \\
&& + \left(F_{\{134589bc, 2, 67, a\}} - F_{\{1345689bc, 2, 7, a\}} - F_{\{134589bc, 2, 6, 7a\}} + F_{\{134589abc, 2, 6, 7\}}\right)
\end{eqnarray*}
\begin{eqnarray*}
&& + \left(F_{\{134589bc, 26, 7, a\}} - F_{\{134589bc, 2, 67, a\}} - F_{\{1234589bc, 6, 7, a\}} + F_{\{1345789bc, 2, 6, a\}}\right) \\
&& + \left(F_{\{123458, 6, 7a, 9bc\}} - F_{\{123458, 6, 7, 9abc\}} - F_{\{1234578, 6, 9bc, a\}} + F_{\{1234589bc, 6, 7, a\}}\right) \\
&& + \left(F_{\{1358b, 2, 49c, 67a\}} - F_{\{13568b, 2, 49c, 7a\}} - F_{\{1358b, 2, 479ac, 6\}} + F_{\{134589bc, 2, 6, 7a\}}\right)\\
&& + \left(F_{\{18, 2, 35b, 4679ac\}} - F_{\{18, 2, 34579abc, 6\}} - F_{\{168, 2, 35b, 479ac\}} + F_{\{1358b, 2, 479ac, 6\}}\right) \\
&& + \left(F_{\{146789ac, 2, 3, 5b\}} - F_{\{168, 2, 3, 4579abc\}} - F_{\{1368, 2, 479ac, 5b\}} + F_{\{168, 2, 35b, 479ac\}}\right) \\
&& + \left(F_{\{12478, 3, 5a, 69bc\}} - F_{\{124578, 3, 69bc, a\}} - F_{\{12478, 3, 5, 69abc\}} + F_{\{1246789bc, 3, 5, a\}}\right) \\
&& + \left(F_{\{1248, 35, 69abc, 7\}} - F_{\{1248, 3, 57, 69abc\}} - F_{\{12348, 5, 69abc, 7\}} + F_{\{12478, 3, 5, 69abc\}}\right) \\
&& + \left(F_{\{123478, 5, 6, 9abc\}} - F_{\{12348, 5, 6, 79abc\}} - F_{\{123468, 5, 7, 9abc\}} + F_{\{12348, 5, 69abc, 7\}}\right) \\
&& + \left(F_{\{1a, 25679bc, 38, 4\}} - F_{\{1, 25679abc, 38, 4\}} - F_{\{138, 25679bc, 4, a\}} + F_{\{1, 2356789bc, 4, a\}}\right) \\
&& + \left(F_{\{138a, 2567, 4, 9bc\}} - F_{\{1235678, 4, 9bc, a\}} - F_{\{138, 2567, 4, 9abc\}} + F_{\{138, 25679bc, 4, a\}}\right) \\
&& + \left(F_{\{1389abc, 2, 4, 567\}} - F_{\{138, 2, 4, 5679abc\}} - F_{\{1238, 4, 567, 9abc\}} + F_{\{138, 2567, 4, 9abc\}}\right) \\
&& + \left(F_{\{12389abc, 4, 56, 7\}} - F_{\{1238, 4, 56, 79abc\}} - F_{\{123568, 4, 7, 9abc\}} + F_{\{1238, 4, 567, 9abc\}}\right) \\
&& + \left(F_{\{125678a, 3b, 4, 9c\}} - F_{\{1235678a, 4, 9c, b\}} - F_{\{125678a, 3, 4, 9bc\}} + F_{\{1256789ac, 3, 4, b\}}\right) \\
&& + \left(F_{\{18a, 2567, 39bc, 4\}} - F_{\{138a, 2567, 4, 9bc\}} - F_{\{18a, 25679bc, 3, 4\}} + F_{\{125678a, 3, 4, 9bc\}}\right) \\
&& + \left(F_{\{1a, 235679bc, 4, 8\}} - F_{\{1a, 25679bc, 38, 4\}} - F_{\{125679abc, 3, 4, 8\}} + F_{\{18a, 25679bc, 3, 4\}}\right) \\
&& + \left(F_{\{12, 3, 48, 5679abc\}} - F_{\{12, 3, 45679abc, 8\}} - F_{\{128, 3, 4, 5679abc\}} + F_{\{125679abc, 3, 4, 8\}}\right) \\
&& + \left(F_{\{1238, 4, 5, 679abc\}} - F_{\{128, 3679abc, 4, 5\}} - F_{\{1258, 3, 4, 679abc\}} + F_{\{128, 3, 4, 5679abc\}}\right)
\end{eqnarray*}

\begin{eqnarray*}
&\equiv&  F_{\{124578ab, 36, 9, c\}} + F_{\{1, 28, 34679abc, 5\}} +
F_{\{1, 2, 345679abc, 8\}} + F_{\{1a, 235679bc, 4, 8\}} \\
&& + F_{\{1234678, 5, 9bc, a\}} + F_{\{1238, 45, 6, 79abc\}} +
F_{\{134589bc, 26, 7, a\}} + F_{\{1, 2368b, 47, 59ac\}} \\
&& + F_{\{1, 234678b, 59c, a\}} + F_{\{125678a, 3b, 4, 9c\}} + F_{\{1248,
  35, 69abc, 7\}} + F_{\{1238, 4, 5, 679abc\}} \\
&& + F_{\{1, 23, 45679abc, 8\}} + F_{\{1234678a, 5, 9c, b\}} + F_{\{15,
  234678ab, 9, c\}} + F_{\{12389abc, 4, 56, 7\}} \\
&& + F_{\{12478, 3, 5a, 69bc\}} + F_{\{146789ac, 2, 3, 5b\}} +
F_{\{1358b, 2, 49c, 67a\}} + F_{\{128, 3, 4579abc, 6\}} \\
&&+ F_{\{18, 2, 35b, 4679ac\}} + F_{\{18a, 2567, 39bc, 4\}} + F_{\{12,
  3, 48, 5679abc\}} + F_{\{123478, 5, 6, 9abc\}} \\
&&+ F_{\{128, 3, 4679abc, 5\}} + F_{\{1, 238, 4579ac, 6b\}} +
F_{\{124578a, 3, 69c, b\}} + F_{\{123458, 6, 7a, 9bc\}} \\
&& + F_{\{18, 2, 34579abc, 6\}} + F_{\{1234578a, 6, 9c, b\}} + F_{\{135678ab, 24, 9, c\}} + F_{\{1389abc, 2, 4, 567\}}.
\end{eqnarray*}

The calculation took 162 seconds.  

Recall that in Example \ref{ex:dihedraln=12} we described a noneffective
expression for this curve class with 103 terms that was obtained using
only simple linear algebra techniques (as opposed to the toric
degeneration techniques used here).  The
\texttt{seekEffectiveExpression} command was also able to find an
effective expression starting from the 103 term expression, but the
calculation took 21210 seconds.  We consider this example as evidence that the toric degeneration method is
superior to the simple linear algebra approach.

%%%%%%%%%%%%%%%%%%%%%%%%%%%%%%%%%%%%%

\end{document}